\documentclass[british,12pt]{article}

\usepackage{babel}

\usepackage{multicol}
\usepackage{enumerate}

\usepackage{amsmath,amssymb,amsthm}

\usepackage{amsmath}
\allowdisplaybreaks
\usepackage{amssymb}
\usepackage{amsthm}

\newtheorem{theorem}{Theorem}[section]
\newtheorem{lemma}[theorem]{Lemma}
\newtheorem{proposition}[theorem]{Proposition}
\newtheorem{corollary}[theorem]{Corollary}

\theoremstyle{definition}
\newtheorem{remark}[theorem]{Remark}

\makeatletter
\def\@map#1#2[#3]{\mbox{$#1 \colon\thinspace #2 \longrightarrow #3$}}
\def\map#1#2{\@ifnextchar [{\@map{#1}{#2}}{\@map{#1}{#2}[#2]}}
\g@addto@macro\@floatboxreset\centering
\makeatother

\newcommand{\torus}{\mathbb{T}^2}
\newcommand{\klein}{\mathbb{K}^2}
\newcommand{\z}{\mathbb{Z}}
\newcommand{\zsz}{\mathbb{Z} \oplus \mathbb{Z}}
\newcommand{\zsdz}{\mathbb{Z} \rtimes \mathbb{Z}}
\newcommand{\ztwo}{\mathbb{Z}_2}
\newcommand{\rtwo}{\mathbb{R}^2}
\newcommand{\id}{\boldsymbol{1}}

\renewcommand{\hom}{{\rm Hom}}

\newcommand{\gsigma}{\overline{\left\langle \sigma^2 \right\rangle}}
\newcommand{\gsigmab}{\overline{\left\langle \sigma^2 \right\rangle}_{\text{Ab}}}
\newcommand{\ab}{{\text{Ab}}}
\renewcommand{\to}{\ensuremath{\longrightarrow}}

\renewcommand{\ker}[1]{\ensuremath{\operatorname{\text{Ker}}\left({#1}\right)}}

\usepackage[x11names]{xcolor}

\usepackage[
bookmarks=true,
breaklinks=true,
bookmarksnumbered=true,
colorlinks= true,
urlcolor= green,
anchorcolor=yellow,
citecolor=blue,
]{hyperref}                   

\makeatletter

\renewcommand{\p@enumii}{}
\makeatother

\begin{document}

\title{The Borsuk-Ulam property for homotopy classes of maps between the torus and the Klein bottle~--~part~2}

\author{DACIBERG LIMA GON\c{C}ALVES
~\footnote{Departamento de Matem\'atica, IME, Universidade de S\~ao Paulo, Rua do Mat\~ao 1010 CEP: 05508-090, S\~ao Paulo-SP, Brazil. 
e-mail: \texttt{dlgoncal@ime.usp.br}}
\and
JOHN GUASCHI
~\footnote{Normandie Univ, UNICAEN, CNRS, LMNO, 14000 Caen, France.
e-mail: \texttt{john.guaschi@unicaen.fr}}
\and
VINICIUS CASTELUBER LAASS
~\footnote{Departamento de Matem\'atica, IME, Universidade Federal da Bahia, Av.\ Adhemar de Barros, S/N Ondina CEP: 40170-110, Salvador-BA, Brazil. 
e-mail: \texttt{vinicius.laass@ufba.br}} 
}


\maketitle

\begin{abstract}%
\noindent
Let $M$ be a topological space that admits a free involution $\tau$, and let $N$ be a topological space. A homotopy class $\beta \in [ M,N ]$ is said to have {\it the Borsuk-Ulam property with respect to $\tau$} if for every representative map $\map{f}{M}[N]$ of $\beta$, there exists a point $x \in M$ such that $f(\tau(x))= f(x)$. In this paper, we determine the homotopy class of maps from the $2$-torus $\torus$ to the Klein bottle $\klein$ that possess the Borsuk-Ulam property with respect to any free involution of $\torus$ for which the orbit space is $\klein$. Our results are given in terms of a certain family of homomorphisms involving the fundamental groups of $\torus$ and $\klein$. This completes the analysis of the Borsuk-Ulam problem for the case $M=\torus$ and $N=\klein$, and for any free involution $\tau$ of $\torus$.  
\end{abstract}

\begingroup
\renewcommand{\thefootnote}{}
\endgroup


\section{Introduction}\label{sec:introduction}

The classical Borsuk-Ulam theorem states that for all $n\in \mathbb{N}$ and any continuous map $\map{f}{\mathbb{S}^n}[\mathbb{R}^n]$, there exists a point $x\in \mathbb{S}^n$ such that $f(-x)=f(x)$~\cite[Satz~II]{Bo}. This result has since been  generalised in many directions, and the reader may consult the extensive survey~\cite{Sten}, the book~\cite{Mato}, as well as the papers~\cite{BiaMat,CotVen,Dol,Fad,FadHus,Izy,IzyJaw,Mar,MarMatSan} (note that this list is by no means exhaustive). One such generalisation consists in the study of the validity of the theorem when we replace $\mathbb{S}^n$ and $\mathbb{R}^n$ by manifolds $M$ and $N$ respectively, and we replace the antipodal map of $\mathbb{S}^n$ by a free involution $\tau$ of $M$. More precisely, the triple $(M,\tau;N)$ is said to have the Borsuk-Ulam property if for any continuous map $\map{f}{M}[N]$, there exists a point $x \in M$ for  which $f(\tau(x))=f(x)$. Some examples of results in this direction may be found in~\cite{BarGonVen,DesPerVen,Gon,GonGua,GonSan}. Very recently, the following more refined Borsuk-Ulam-type problem was introduced by the authors in the context of homotopy classes of maps from $M$ to $N$~\cite{GonGuaLaa1}. If $\beta\in [M,N]$ is a homotopy class of maps between $M$ and $N$, $\beta$ is said to have the Borsuk-Ulam property with respect to $\tau$ if for every map $f\in \beta$, there exists a point $x\in M$ (that depends on $f$) such that $f(\tau(x))=f(x)$. If a triple $(M, \tau; N)$ satisfies the Borsuk-Ulam property, then it is certainly the case that every homotopy class of maps between $M$ and $N$ satisfies the Borsuk-Ulam property with respect to $\tau$. The study of the converse leads to an interesting and delicate  question that was posed in~\cite{GonGuaLaa1}, namely the classification of those elements of $[M,N]$ that satisfy the Borsuk-Ulam property with respect to the possible  free involutions $\tau$ of $M$. In that paper, the authors solved this problem in the cases where $M=N$ and $M$ is either the $2$-torus $\torus$ or the Klein bottle $\klein$. It is then natural to consider the case $M=\torus$ and $N=\klein$. In this case, if $\tau$ is a free involution of $\torus$ then the corresponding orbit space $\torus/\langle\tau\rangle$ is either $\torus$ and $\klein$. In the first case, where $\torus/\langle\tau\rangle=\torus$, the authors recently determined the elements of the set $[\torus,\klein]$ that possess the Borsuk-Ulam property with respect to $\tau$~\cite{GonGuaLaa2}. The current paper is a continuation of~\cite{GonGuaLaa2}, in the sense that we determine the elements of the set $[\torus,\klein]$ that possess the Borsuk-Ulam property with respect to $\tau$ in the second case, where $\torus/\langle\tau\rangle=\klein$. In each of the two cases, by~\cite[Proposition~21]{GonGuaLaa1} there is only one class of free involutions, and by~\cite[Proposition~8]{GonGuaLaa1}, it suffices to consider a specific free involution of $\torus$. 

 
In order to state Theorem~\ref{th:BORSUK_TAU_2}, which is the main result of this paper, we first recall some facts and notation. As in~\cite[Theorems~12 and~19]{GonGuaLaa1}, we identify $\pi_1 (\torus,\ast )$ and $\pi_1 (\klein, \ast)$ with the free Abelian group $\zsz$ and the (non-trivial) semi-direct product $\zsdz$ respectively. Consider the following short exact sequence:
\begin{align}
& 1 \to \pi_1(\torus)= \zsz \stackrel{i_2}{\longrightarrow}  \pi_1(\klein)= \zsdz \stackrel{\theta_2}{\longrightarrow} \ztwo \to 1,\label{eq:homo_tau_2}
\end{align}
where the homomorphisms $i_{2}$ and $\theta_{2}$ are defined by:
\begin{equation*}
\text{$i_2\colon\thinspace\begin{cases}
	 (1,0) \longmapsto (1,0)\\
	 (0,1) \longmapsto (0,2)
	\end{cases}$ and 
$\theta_2\colon\thinspace\begin{cases}
	(1,0) \longmapsto \overline{0}\\
	(0,1) \longmapsto \overline{1}.
	\end{cases}$}
\end{equation*}
By standard results in covering space theory, there exists a double covering $\map{c_2}{\torus}[\klein]$ whose induced homomorphism on the level of fundamental groups is $i_2$. If $\map{\tau_2}{\torus}$ is the non-trivial deck transformation associated with $c_2$, then it is a free involution. Further, $\tau_{2}$ lifts to a homeomorphism $\map{\widehat{\tau}_{2}}{\rtwo}$, where $\widehat{\tau}_{2}(x,y)=(x+\frac{1}{2},1-y)$ for all $(x,y)\in \rtwo$.

We recall an appropriate algebraic description of the set $[\torus,\klein]$ that was given in~\cite[Proposition~1.1 and Remark~1.2]{GonGuaLaa2}.

\begin{proposition}\label{prop:set_homotopy}
The set $[\torus,\klein]$ is in bijection with the subset of $\hom(\zsz,\zsdz)$ whose elements are described as follows:	
\begin{multicols}{2}
\setlength{\parskip}{0.2\baselineskip}

Type 1: ${\allowdisplaybreaks
\begin{cases}
 (1,0) \longmapsto (i,2s_1+1)\\
 (0,1) \longmapsto (0,2 s_2 )
\end{cases}}$

Type 3: ${\allowdisplaybreaks
\begin{cases}
 (1,0) \longmapsto (0,2s_1)\\
 (0,1) \longmapsto (i,2 s_2 +1)
\end{cases}}$

Type 2: ${\allowdisplaybreaks
\begin{cases}
 (1,0) \longmapsto (i,2s_1+1)\\
 (0,1) \longmapsto (i,2 s_2+1)
\end{cases}}$

Type 4: ${\allowdisplaybreaks
\begin{cases}
 (1,0) \longmapsto (r_1,2s_1)\\
 (0,1) \longmapsto (r_2,2 s_2 ),
\end{cases}}$
\end{multicols}
\noindent where $i \in \{0,1 \}$ and $s_1, s_2 \in \z$ for Types~1,~2 and~3, and $r_1, r_2, s_1, s_2 \in \z$ and $r_1 \geq 0$ for Type~4.
\end{proposition}

\begin{remark}\label{rem:homotopy_pi1}
The bijection of Proposition~\ref{prop:set_homotopy} may be obtained using standard arguments in homotopy theory that are described in detail in~\cite[Chapter~V, Corollary~4.4]{White}, and more briefly in~\cite[Theorem~4]{GonGuaLaa1}. Within the framework of this paper, it may be defined as follows: given a homotopy class $\beta \in [\torus,\klein]$, there exists a pointed map $\map{f}{(\torus,\ast)}[(\klein,\ast)]$ that gives rise to a representative of $\beta$ if we omit the basepoints. The induced homomorphism $\map{f_\#}{\pi_1(\torus,\ast)}[\pi_1(\klein,\ast)]$ is conjugate to exactly one of the elements of $\hom(\zsz,\zsdz)$, denoted by $\beta_\#$, and described in Proposition~\ref{prop:set_homotopy}. Note that $\beta_\#$ is independent of the choice of $f$.
\end{remark}

The following theorem is the main result of this paper.

\begin{theorem}\label{th:BORSUK_TAU_2}
Let $\beta \in [\torus,\klein]$, and let $\beta_\# \in \hom(\zsz,\zsdz)$. Then $\beta$ has the Borsuk-Ulam property with respect to $\tau_2$ if and only if one of the following conditions is satisfied:
\begin{enumerate}[(a)]
\item $\beta_\#$ is a homomorphism of Type~1, and $s_2$ is even.
\item $\beta_\#$ is a homomorphism of Type~2.
\item $\beta_\#$ is a homomorphism of Type~3, and $s_1 \neq 0$.
\item $\beta_\#$ is a homomorphism of Type~4, and one of the following conditions holds: 
\begin{enumerate}[(i)]
\item $r_2  s_1 \neq 0$.
\item $r_2 = s_2 = 0$ and $s_1 \neq 0$.
\item $s_1 = s_2 = 0$, $r_1 \neq  0$ and $r_2$ is even.
\end{enumerate}
\end{enumerate}
%
%
\end{theorem}


It follows from Theorem~\ref{th:BORSUK_TAU_2} and the remarks in the first paragraph regarding~\cite[Propositions~8 and~21]{GonGuaLaa1} that if $\tau$ is an arbitrary free involution of $\torus$, one may decide which elements of the set $[\torus,\klein]$ possess the Borsuk-Ulam property with respect to $\tau$, which solves the Borsuk-Ulam problem for $[\torus,\klein]$.


One of the main tools used in this paper is the study of a certain two-variable equation in the $2$-string pure braid group of the Klein bottle, as well as some additional information that may be obtained from the fundamental groups of the torus and the Klein bottle. So the solutions of certain equations in the braid groups of some of the surfaces in question play an important r\^ole in the resolution of the Borsuk-Ulam problem for homotopy classes. These equations are derived from a commutative diagram involving fundamental groups and $2$-string braid groups of the surfaces (see~\cite[Theorem~7]{GonGuaLaa1} for more details).

The rest of this paper comprises two sections. In Section~\ref{sec:equalg}, we start by recalling some notation and a number of previous results. In Proposition~\ref{prop:p2_k2}, we describe some relevant properties of the $2$-string pure braid group $P_{2}(\klein)$ of the Klein bottle that appeared in~\cite{GonGuaLaa2}. In Lemma~\ref{lem:algebra_tau_2}, we give an algebraic criterion involving elements of $P_{2}(\klein)$ for a homotopy class to satisfy the Borsuk-Ulam property, and in Lemma~\ref{lem:lemma_1}, we derive a useful necessary condition, in terms of the existence of solutions to a certain equation in a free Abelian group of infinite rank, for a given homotopy class to satisfy this property. Section~\ref{sec:borsuk_2} of the paper is devoted to proving Theorem~\ref{th:BORSUK_TAU_2}. The proof will follow from Propositions~\ref{prop:tau_2_case_23I}--\ref{prop:tau_2_case_6I} whose statements correspond to the types of homotopy classes given by Proposition~\ref{prop:set_homotopy}.


\section{Preliminaries and algebraic criteria}\label{sec:equalg}

%

Let $\alpha=[f] \in [ \torus,\ast ; \klein,\ast]$ be a pointed homotopy class, let $\beta \in [ \torus,\klein]$ be the homotopy class for which $f$ is a representative map if we omit the basepoints, and let $\map{\tau_2}{\torus}$ be the free involution defined in the Introduction. By~\cite[Theorem~7(b)]{GonGuaLaa1}, $\alpha$ has the Borsuk-Ulam property with respect to the free involution $\tau_2$ if and only if $\beta$ does. So to prove Theorem~\ref{th:BORSUK_TAU_2}, it suffices to restrict our attention to pointed homotopy classes. Let $\alpha_\#$ denote the induced homomorphism $\map{f_\#}{\pi_1 (\torus , \ast)}[\pi_1 (\klein , \ast)]$. We will make use of the following properties of the $2$-string pure braid group $P_2 (\klein)$ of $\klein$ that were derived in~\cite[Section~3]{GonGuaLaa2}.
%
\begin{proposition}\cite[Propositions~3.1, 3.3 and~3.5]{GonGuaLaa2}\label{prop:p2_k2}
The group $P_2(\klein)$ is isomorphic to the semi-direct product $F(u,v) \rtimes_\theta (\zsdz)$, where $F(u,v)$ is the free group of rank $2$ on the set $\{u,v\}$, and the action $\map{\theta}{\zsdz}[\operatorname{Aut}(F(u,v))]$ is defined as follows:
\begin{equation*}
\theta(m,n): \begin{cases}
u \longmapsto B^{m-\delta_n} u^{\varepsilon_n} B^{-m+\delta_n}\\
v \longmapsto B^m v u^{-2m} B^{-m+\delta_n}\\
B  \longmapsto B^{\varepsilon_n},
\end{cases} 
\end{equation*}
where $\delta_n=\begin{cases}
0 & \text{if $n$ is even}\\
1 & \text{if $n$ is odd,}
\end{cases}$ $\varepsilon_n=(-1)^n$ and $B=uvuv^{-1}$. With respect to this decomposition, the following properties hold:
\begin{itemize}
\item the standard Artin generator $\sigma \in B_2(\klein)$ satisfies $\sigma^2 =(B ; 0,0)$.
	
\item if $\map{l_\sigma}{P_2(\klein)}$ is the homomorphism defined by $l_\sigma(b)=\sigma b \sigma^{-1}$ for all $b \in P_2 (\klein)$, then:
\begin{align*}
l_\sigma(u^r;0,0)&=((Bu^{-1})^r B^{-r} ; r,0) & l_\sigma(\id;m,0)&=(\id ; m,0)\\
l_\sigma(v^s;0,0)&=((uv)^{-s} (u B)^{\delta_s} ; 0,s) & l_\sigma(\id;0,n) &=(B^{\delta_n} ; 0,n)\\
l_\sigma(B;0,0)&=(B ; 0,0) &&
\end{align*}
for all $m,n,r,s \in \z$, where the symbol $\id$ denotes the trivial element of $F(u,v)$.

\item if $\map{p_1}{F_2 (\klein)}[\klein]$ is the map defined by $p_1(x,y)=x$, then the induced homomorphism $\map{(p_1)_\#}{P_2(\klein)}[\pi_1(\klein)=\zsdz]$ satisfies $(p_1)_\# (w; r,s)=(r,s)$.
\end{itemize}
Given an element $w \in F(u,v)$, let $\rho(w) \in F(u,v)$ and $g(w) \in \zsdz$ such that $l_\sigma(w ; 0, 0)=(\rho(w),g(w))$. Then $\map{g}{F(u,v)}[\zsdz]$ is the homomorphism defined on the basis $\{ u , v \}$ by:
\begin{equation*}
\begin{cases}
g(u)=(1,0)\\
g(v)=(0,1).
\end{cases}
\end{equation*}
Let $\gsigma$ be the normal closure of the element $\sigma^2$. 
Up to isomorphism, $\gsigma$ may be identified with the group $\ker g$ which is the free group of infinite countable rank on the set $\{ B_{k,l} \}_{k,l\in \z}$, where $B_{k,l} =v^k u^l B u^{-l} v^{-k}$ for all $k,l \in \z$.
%
%
With respect to this description, the action $\map{\theta}{\zsdz}[\operatorname{Aut}F(u,v)]$ and the map $\map{\rho}{F(u,v)}$ induce homomorphisms $\zsdz \to \gsigma$ and $\gsigma \to \gsigma$ respectively, which we also denote by $\theta$ and $\rho$ respectively. Let $w \in F(u,v)$, and let $g(w)=(r,s)$. Then there exists a unique element $x \in \gsigma$ such that $w=u^r v^s x$.
\end{proposition}

The following algebraic criterion, similar to that of~\cite[Lemma~23]{GonGuaLaa1}, will be used in what follows to decide whether a pointed homotopy class possesses the Borsuk-Ulam property with respect to $\tau_2$.

\begin{lemma}\label{lem:algebra_tau_2}
A pointed homotopy class $\alpha \in [\torus,\ast ; \klein, \ast]$ does not have the Borsuk-Ulam property with respect to $\tau_2$ if and only if there exist $a,b \in P_2 (\klein)$ such that:
\begin{enumerate}[(i)]

\item\label{eq:algebra_i} $a b l_\sigma(a)=b$.

\item\label{eq:algebra_ii} $(p_1)_\# (a)= \alpha_\# (1,0)$.

\item\label{eq:algebra_iii} $(p_1)_\# (b l_\sigma(b))=\alpha_\# (0,1)$.
\end{enumerate}
\end{lemma}

\begin{proof}
The proof is similar to that of~\cite[Lemma~23]{GonGuaLaa1}, using Proposition~\ref{prop:p2_k2} instead of~\cite[Theorem~12]{GonGuaLaa1}, and the details are left to the reader.
\end{proof}

\begin{corollary}\label{cor:reduction_tau_2}
Let $\alpha , \alpha' \in [ \torus , \ast ; \klein , \ast]$ be pointed homotopy classes, and suppose that:
\begin{equation*}
\text{$\alpha_\#: \begin{cases}
	 (1,0) \longmapsto (r_1 , s_1)\\
	 (0,1) \longmapsto (r_2 , s_2)
	\end{cases}$
and $\alpha'_\#: \begin{cases}
	 (1,0) \longmapsto (r_1 , s_1)\\
	 (0,1) \longmapsto (r_2 , s_2') 
	\end{cases}$}
\end{equation*}
for some $r_1, r_2 , s_1 , s_2 , s_2' \in \z$. If $s_2 \equiv s_2' \bmod 4$ then $\alpha$ has the Borsuk-Ulam property with respect to $\tau_2$ if and only if $\alpha'$ has the Borsuk-Ulam property with respect to $\tau_2$.
	\end{corollary}
	
\begin{proof}
Since the statement is symmetric with respect to $\alpha$ and $\alpha'$, it suffices to show that if $\alpha$ does not have the Borsuk-Ulam property then neither does $\alpha'$. If $\alpha$ does not have the Borsuk-Ulam property, there exist $a,b \in P_2(\klein)$ satisfying~(\ref{eq:algebra_i})--(\ref{eq:algebra_iii}) of Lemma~\ref{lem:algebra_tau_2}. By hypothesis, there exists $k \in \z$ such that $s'_2=s_2+4 k$. Let $b'=b (\id ; 0 , 2 k)$. As in the proof of~\cite[Corollary~4.2]{GonGuaLaa2}, the centre of $B_2(\klein)$ is generated by $(\id ; 0 , 2)$, so $ab' l_\sigma(a)=a b l_\sigma(a) (\id ; 0 , 2 k)=b(\id ; 0 , 2 k)=b'$ by~(\ref{eq:algebra_i}), $(p_1)_\# (a)= \alpha_\# (1,0)=\alpha'_\# (1,0)$ by~(\ref{eq:algebra_ii}), and:
%
%
%
\begin{align*}
(p_1)_\# (b' l_\sigma(b')) &=  (p_1)_\# (b(\id ; 0 , 2 k) l_\sigma(b (\id ; 0 , 2 k)))
= (p_1)_\# (b l_\sigma(b)) (p_1)_\# (\id; 0 , 4k) \\
& \stackrel{\text{(\ref{eq:algebra_iii})}}{=}  (r_2,s_2)(0, 4 k)=(r_2 , s_2')=\alpha'_\# (0,1).
	\end{align*}
Lemma~\ref{lem:algebra_tau_2} implies that $\alpha'$ does not have the Borsuk-Ulam property with respect to $\tau_2$, from which the result follows.
	\end{proof}

In addition to Proposition~\ref{prop:p2_k2}, we use some facts and notation about some automorphisms and elements of $\gsigma$ that are summarised in the following proposition. If $l\in \z$, let $\sigma_{l}$  denote its sign, \emph{i.e}\ $\sigma_{l}=1$ if $l>0$, $\sigma_{l}=-1$ if $l<0$, and $\sigma_{l}=0$ if $l=0$, and if $x$ and $y$ are elements of a group then let $[x,y]=xyx^{-1}y^{-1}$ denote their commutator.

\begin{proposition}\cite[equations~(3.14)--(3.16) and Proposition~3.7]{GonGuaLaa2}\label{prop:gsigma} 
The group $\gsigmab$ is free Abelian, and $\{ B_{k,l} := v^k u^l B u^{-l} v^{-k}\; | \; k,l\in \z \}$ is a basis, namely:
\begin{equation*}
\gsigmab= \bigoplus_{k,l \in \z} \z \left[ B_{k,l} \right].
\end{equation*}
Let $p,q \in \z$ and consider the following automorphism of $\gsigma$:
\begin{equation}\label{eq:defcpq}
\begin{array}{rrcl}
c_{p,q}: & \gsigma & \longrightarrow & \gsigma\\
& x & \longmapsto & v^p u^q x u^{-q} v^{-p}.
\end{array}
\end{equation}
For all $(m,n) \in \zsdz$ and $p,q \in \z$, the endomorphisms $\theta(m,n), \rho$ and $(c_{p,q})$ of $\gsigma$ induce endomorphisms $\theta(m,n)_\ab, \rho_\ab$ and $(c_{p,q})_\ab$ of $\gsigmab$ respectively, and they satisfy:
\begin{align}
\theta(m,n)_\ab(B_{k,l}) &= { \varepsilon_n}B_{k ,\varepsilon_n l-2 \delta_k m} \label{eq:homo_gsigmab_theta} \\
\rho_\ab(B_{k.l}) &={\varepsilon_k} B_{-k,\varepsilon_{(k+1)} l}, \text{ and} \label{eq:homo_gsigmab_rho} \\
(c_{p,q})_\ab(B_{k,l}) &= B_{k+p, l+\varepsilon_{k} q}. \label{eq:homo_gsigmab_cpq}
\end{align}
If $k,l \in \z$ and $r \in \{ 0 , 1 \}$, consider the following elements of $F(u,v)$:
\begin{equation}\label{eq:defTIOJ}
\text{$T_{k,r}= u^{k} (B^{\varepsilon_{r}} u^{-\varepsilon_{r}})^{k\varepsilon_{r}}$, $I_{k}=v^{k} (v B )^{-k}$, $O_{k,l}=\left[ v^{2k},u^l \right]$ and $J_{k,l}=v^{2k}(v u^l  )^{-2k}$.}
\end{equation}
%
%
Then $T_{k,r}, I_{k}, O_{k,l}$ and $J_{k,l}\in \gsigma$. Let $\widetilde{T}_{k,r}$, $\widetilde{I}_k$, $\widetilde{O}_{k,l}$ and $\widetilde{J}_{k,l}$ be the projections of $T_{k,r}, I_{k}, O_{k,l}$ and $J_{k,l}$ in $\gsigmab$. 
\begin{enumerate}[(a)]
\item\label{it:wordsb}  If $k=0$ then $\widetilde{T}_{0,r}=\widetilde{I}_0=0$, and if $k=0$ or $l=0$ then $\widetilde{O}_{k,l}=\widetilde{J}_{k,l}=0$. 

\item\label{it:wordsc} For all $k,l \neq 0$:
\begin{align*}
\hspace*{-6mm}\widetilde{I}_k&=- \sigma_{k} \sum_{i=1}^{\sigma_k k} B_{\sigma_k  i+(1-\sigma_k )/2 ,0}  & \widetilde{O}_{k,l}&=\sigma_k \sigma_l \sum_{i=1}^{\sigma_k k} \sum_{j=1}^{\sigma_l l} \bigl(B_{\sigma_k (2i-1 ),-\sigma_l j+(\sigma_l -1 )/2} 
- B_{ \sigma_k (2i-1) -1 , \sigma_l j-(1+\sigma_l)/2}
\bigr)\\
\hspace*{-6mm}\widetilde{T}_{k,r}&= \sigma_{k} \sum_{i=1}^{\sigma_{k}k} B_{0,\sigma_{k}(i+(\sigma_{k}(1-2r)-1)/2)} & \widetilde{J}_{k,l}&=-\sigma_{k}\sigma_{l} \sum_{i=1}^{\sigma_{k}k} \sum_{j=1}^{\sigma_{l}l} B_{\sigma_{k}(2i-1), \sigma_{l}(j-(1+\sigma_{l})/2)}.
%
\end{align*}
\end{enumerate}
\end{proposition}

If $k,l \in \z$, let:
\begin{equation}\label{eq:defQ}
Q_{k,l}=u^k v^{2l+1} u^k v^{-2l-1} \in F(u,v).
\end{equation}

\begin{proposition}\label{prop:element_q}
Let $k,l \in \z$. Then $Q_{k,l} \in \gsigma$, and $Q_{0,l}=\id$. If $\widetilde{Q}_{k,l}$ denotes the projection of $Q_{k,l}$ in $\gsigmab$ then for all $k\neq 0$:
\begin{equation}\label{eq:formqkl}
Q_{k,l}=O_{l,k}^{-1} \left(\prod_{i=1}^{\sigma_k k} B_{2l,-i+k (1+\sigma_k)/2}\right)^{\sigma_k}
\end{equation}
 and 
\begin{equation}\label{eq:formqkl2}
 \widetilde{Q}_{k,l} =-\widetilde{O}_{l,k}+\sigma_k \sum_{i=1}^{\sigma_k k} B_{2l,\sigma_k i-(1+\sigma_k)/2}. 
\end{equation}
\end{proposition}

\begin{proof}
Clearly $Q_{0,l}=\id$ for all $l\in \z$. So assume that $k\neq 0$, and suppose first that $l=0$. If $\lvert k\rvert=1$ then $Q_{k,0}=u^{k} v u^{k} v^{-1}=B_{0,-1+(1+\sigma_{k})/2}^{\sigma_{k}}$, and~(\ref{eq:formqkl}) is valid in this case.
Suppose then that~(\ref{eq:formqkl}) holds for some $k\neq 0$. Then by induction we have: 
\begin{align*}
Q_{k+\sigma_{k},0} &= u^{\sigma_{k}} u^{k}vu^{k}v^{-1} u^{-\sigma_{k}} \ldotp u^{\sigma_{k}} v u^{\sigma_{k}} v^{-1}= c_{0,\sigma_{k}} (Q_{k,0})  \ldotp B_{0,-1+(1+\sigma_{k})/2}^{\sigma_{k}}\\
& =\left(\prod_{i=1}^{\sigma_k k} B_{0,-i+\sigma_k+k(1+\sigma_k)/2}\right)^{\sigma_k} B_{0,-1+(1+\sigma_{k})/2}^{\sigma_{k}}\\
&=\begin{cases}
\left(\displaystyle\prod_{i=1}^{k} B_{0,-i+1+k}\right) B_{0,0}= \displaystyle\prod_{i=1}^{k+1} B_{0,-i+1+k} & \text{if $k>0$}\\
\left(\displaystyle\prod_{i=1}^{-k} B_{0,-i-1}\right)^{-1} B_{0,-1}^{-1}=\left(\displaystyle\prod_{i=1}^{-(k-1)} B_{0,-i}\right)^{-1} & \text{if $k<0$.}
\end{cases}
\end{align*}
Thus~(\ref{eq:formqkl}) holds for all $k\neq 0$ and $l=0$. Finally, suppose that $k \neq 0$ and $l\neq 0$. Then from the case $l=0$, we have:
\begin{align*}
Q_{k,l} & =(u^k v^{2l}  u^{-k} v^{-2l}) v^{2l} (u^k  v u^k v^{-1}) v^{-2l}=O_{l,k}^{-1} c_{2l,0} (Q_{k,0})=O_{l,k}^{-1} c_{2l,0} \left(\prod_{i=1}^{\sigma_k k} B_{0,-i+k (1+\sigma_k)/2 }\right)^{\sigma_k}\\
&= O_{l,k}^{-1} \left(\prod_{i=1}^{\sigma_k k} B_{2l,-i+k (1+\sigma_k)/2 }\right)^{\sigma_k}.
\end{align*}
and (\ref{eq:formqkl}) holds in this case. So for all $k,l\in \z$, (\ref{eq:formqkl}) is valid, and $Q_{k,l} \in \gsigma$. Equation~(\ref{eq:formqkl2}) then follows by projecting into $\gsigmab$ and using the fact that the sets $\{ -i+k(1+\sigma_k)/2 \, | \, 1 \leq i \leq \sigma_k k \}$ and $\{ \sigma_k i - (1+\sigma_k)/2 \, | \, 1 \leq i \leq \sigma_k k \}$ are equal.
\end{proof}

Let $a,b\in P_{2}(\klein)$. By~\cite[Lemma~4.8]{GonGuaLaa2}, there exist $a_1,a_2,b_1,b_2, m_1,n_1,m_{2}, n_{2} \in \z$ and $x,y \in \gsigma$ such that:
\begin{equation}\label{eq:normformab}
\text{$a=(u^{a_1} v^{a_2} x ; m_1 , n_1)$ and $b=(u^{b_1} v^{b_2} y ; m_{2} , n_{2})$.}
\end{equation}
Exchanging the r\^oles of $a$ and $b$ in~\cite[equations~(4.5),~(4.6) and~(4.8)]{GonGuaLaa2}, and noting that $a_{2}$ is not necessarily zero (as it was in the proof of~\cite[Lemma~4.8]{GonGuaLaa2}), it follows that:\\
\begin{align}
bl_{\sigma}(a) 
=& (u^{b_1} v^{b_2} y \theta(m_{2} , \delta_{n_{2}})((Bu^{-1})^{a_{1}} B^{-a_{1}} \theta(a_{1},0)((uv)^{-a_{2}}(uB)^{\delta_{a_{2}}}) \theta(a_{1},\delta_{a_{2}})(\rho(x) B^{\delta_{n_{1}}}));\notag\\
& m_{2}+ \varepsilon_{n_{2}}(a_{1}+(-1)^{\delta_{a_{2}}}m_{1}), a_{2}+n_{1}+n_{2})\notag\\
=& (u^{b_1} v^{b_2} y \theta(m_{2} , \delta_{n_{2}})((Bu^{-1})^{a_{1}} B^{-a_{1}})\theta(m_{2}+ \varepsilon_{n_{2}}a_{1} , \delta_{n_{2}})((uv)^{-a_{2}}(uB)^{\delta_{a_{2}}}) \notag\\ 
&\theta(m_{2}+\varepsilon_{n_{2}}a_{1} , \delta_{n_{2}}+\delta_{a_{2}})(\rho(x) B^{\delta_{n_{1}}}); m_{2}+ \varepsilon_{n_{2}}(a_{1}+(-1)^{\delta_{a_{2}}}m_{1}), a_{2}+n_{1}+n_{2}).\label{eq:blsiga}
\end{align}
In a similar manner, exchanging the r\^oles of $a$ and $b$ in~\cite[equation~(4.4)]{GonGuaLaa2}, replacing $b$ by $a$ in~\cite[equation~(4.5)]{GonGuaLaa2}, and then substituting $a$ (resp.\ $b$) by $ab$ (resp.\ $a$) in~\cite[equation~(4.6)]{GonGuaLaa2}, we see that:
\begin{align}\label{eq:a_b_lsigma_a1}
abl_\sigma(a) 
= & (u^{a_1}v^{a_2} x \theta(m_1,\delta_{n_1})(u^{b_1}v^{b_2}y) \theta(m_1+\varepsilon_{n_1}m_2,\delta_{n_1+n_2})((Bu^{-1})^{a_{1}} B^{-a_{1}} \theta(a_1,0)((uv)^{-a_2}(uB)^{\delta_{a_2}} \nonumber \\
& \theta(0,\delta_{a_2})(\rho(x)B^{\delta_{n_{1}}})));  m_1+\varepsilon_{n_1}m_2 + \varepsilon_{n_1+n_2} (a_1 + \varepsilon_{a_2} m_1), 2n_1+n_2+a_2).
\end{align}

The following result is similar to~\cite[Lemma~4.8]{GonGuaLaa2}, and will be used in the proof of Lemma~\ref{lem:lemma_1}.

\begin{lemma}\label{lem:lem48ggl2}
Let $a,b \in P_2 (\klein)$, which we write in the form~(\ref{eq:normformab}). 
Suppose that $a$ and $b$ satisfy the relation of Lemma~\ref{lem:algebra_tau_2}(\ref{eq:algebra_i}). 
Then:
\begin{equation}\label{eq:a1a2}
\text{$a_{1}=\varepsilon_{n_{2}}m_{2}(\varepsilon_{n_{1}}-1)-m_{1}(1+\varepsilon_{n_{1}+n_{2}})$ and $a_{2}=-2n_{1}$,}
\end{equation}
so $a_{1}$ and $a_{2}$ are even, and:
\begin{align}
y =& v^{-b_{2}} u^{a_{1}-b_{1}} v^{a_{2}} x \theta(m_{1},\delta_{n_{1}})(u^{b_{1}}v^{b_{2}}y)\ldotp\notag\\
&\theta(m_{1}+(-1)^{\delta_{n_{1}}}m_{2},n_{1}+n_{2})((Bu^{-1})^{a_{1}} B^{-a_{1}} \theta(a_{1},0)((Bv^{2})^{-a_{2}/2}) \theta(a_{1},0)(\rho(x) B^{\delta_{n_{1}}})).\label{eq:yexp}
\end{align}
\end{lemma}

\begin{proof}
Let $a,b \in P_2 (\klein)$, which we write in the form~(\ref{eq:normformab}).
By Lemma~\ref{lem:algebra_tau_2}(\ref{eq:algebra_i}), we have $(p_{1})_{\#}(ab l_{\sigma}(a))=(p_{1})_{\#}(b)$. By~(\ref{eq:normformab}) and~(\ref{eq:a_b_lsigma_a1}), we thus obtain the second relation of~(\ref{eq:a1a2}), and $m_{2}=m_{1}+ \varepsilon_{n_{1}}m_{2} +\varepsilon_{n_{1}+n_{2}} (a_{1}+ m_{1})$,
where we have used the fact that $a_{2}$ is even, and that $(-1)^{\delta_{q}}=\varepsilon_{q}$ for all $q\in \z$. The first relation of~(\ref{eq:a1a2}) then follows, and we deduce also that $a_{1}$ is even. Equation~(\ref{eq:yexp}) is then also a consequence of~(\ref{eq:normformab}) and~(\ref{eq:a_b_lsigma_a1}), using also the equality $(uv)^{2}=Bv^{2}$ and the fact that $a_{2}$ is even.
\end{proof}

The following lemma  will be used in the proofs of  Propositions~\ref{prop:tau_2_case_23I},~\ref{prop:tau_2_case_5I} and~\ref{prop:tau_2_case_7I}. 
	
\begin{lemma}\label{lem:lemma_1}
Let $\alpha \in [\torus, \ast;\klein, \ast]$ and suppose that ${\allowdisplaybreaks
	\alpha_\#: \begin{cases}
	 (1,0) \longmapsto (\delta_{i+1} \delta_{j+1} r_1 ,  2s_1+i) \\
	 (0,1) \longmapsto (\delta_{i+1} \delta_{j+1} r_2 , 2s_2+j),
	\end{cases}
}$ where $r_1,r_2,s_1,s_2 \in \z$ and $i,j \in \{0,1\}$. If $\alpha$ does not have the Borsuk-Ulam property with respect to $\tau_2$, then there exist $m,n \in \z$ and $x,y \in \gsigmab$ for which the following equality holds in $\gsigmab$:
\begin{multline}\label{eq:lema_1}
(c_{a_2-b_2,a_1-b_1})_\ab (x)+\theta(g,\delta_{n+i})_\ab (\rho_\ab (x))+(c_{a_2,a_1\varepsilon_{n+i}})_\ab (\theta(\delta_{i+1}\delta_{j+1}r_1,\delta_i)_\ab(y))-y\\
+(c_{a_2,0})_\ab(\widetilde{T}_{a_1\varepsilon_{n+i},\delta_{n+i}} )
+(c_{a_2-b_2,0})_\ab (\widetilde{O}_{2s_1+i,a_1-b_1}-\delta_{j+1} \widetilde{O}_{s_{2}-n,2\delta_i m-2\delta_{i+1}\delta_{n+1}r_1}+\delta_j \widetilde{Q}_{-2\delta_i m,s_{2}-n}) \\
+(c_{a_2,a_1\varepsilon_{n+i}})_\ab (\widetilde{J}_{\delta_{i+1}(n-s_2),-2\delta_{i+1}\delta_{j+1}r_1}) 
+(c_{a_2-1,a_1 \varepsilon_{n+i+1}})_\ab(\widetilde{I}_{-\delta_i b_2}) 
+(c_{0,\delta_{n+i+1}})_\ab (\widetilde{J}_{-2s_1 -i,1-2g})  \\
+\widetilde{O}_{-2s_1-i,\delta_{n+i-1}} +(\delta_{n+i}+\delta_i \varepsilon_{n+i}-g)B_{0,0}+(\delta_i-\delta_{n+i}+\varepsilon_i m)B_{a_2,a_1 \varepsilon_{n+i}}+(\delta_{i+1}\delta_{j+1}r_1-\delta_i)B_{a_2-b_2,a_1-b_1}
= 0,
\end{multline}
where $a_1=-2(\delta_{i+1}\delta_{j+1}\delta_{n+1} r_1+\delta_i \varepsilon_{n} m)$, $a_2=-4s_1-2i$, 
$b_1 =\delta_{i+1}\delta_{j+1}\varepsilon_{n} r_2+2 \delta_{j+n+1}\varepsilon_{j+1} m$, $b_2=2s_2-2n+j$ 
and  $g=\delta_{i+1}\delta_{j+1}r_1+\varepsilon_i m+\varepsilon_{n+i} a_1$. 
\end{lemma}

\begin{proof}
Suppose that $\alpha \in [\torus, \ast;\klein, \ast]$ does not have the Borsuk-Ulam property with respect to $\tau_2$. Then there exist $a,b \in P_2 (\klein)$ such that conditions~(\ref{eq:algebra_i})--(\ref{eq:algebra_iii}) of Lemma~\ref{lem:algebra_tau_2} hold. With the notation of Lemma~\ref{lem:lem48ggl2}:
\begin{equation}\label{eq:defm1n1}
(m_1,n_1)=(\delta_{i+1}\delta_{j+1}r_1,2s_1+i)
\end{equation} 
by Lemma~\ref{lem:algebra_tau_2}(\ref{eq:algebra_ii}), and $(\delta_{i+1}\delta_{j+1}r_2,2s_2+j)=(m_{2},n_{2}) (b_1,b_2)(m_{2},n_{2})$ by Lemma~\ref{lem:algebra_tau_2}(\ref{eq:algebra_iii}) and Proposition~\ref{prop:p2_k2}, so:
\begin{equation}\label{eq:defb2}
b_2=-2n_{2}+2s_2+j
\end{equation}
by comparing the second coordinates. Thus:
%
\begin{align}
(\delta_{i+1}\delta_{j+1}r_2,2s_2+j)&=(m_{2},n_{2}) (b_1,- 2n_{2} +2s_2+j)(m_{2},n_{2})= (m_{2}+\varepsilon_{n_{2}} b_1 , -n_{2}+2s_2+j) (m_{2},n_{2})\notag\\
&= (m_{2}+\varepsilon_{n_{2}} b_1+\varepsilon_{-n_{2}+j} m_2 , 2s_2+j).\label{eq:imp1blsb}
\end{align}
The first coordinate of~(\ref{eq:imp1blsb}) yields $\varepsilon_{n_{2}} b_1=\delta_{i+1}\delta_{j+1}r_2-m_{2}-\varepsilon_{-n_{2}+j} m_{2}$, and so:
%
%
%
\begin{equation}\label{eq:defb1}
b_1 =\delta_{i+1}\delta_{j+1}\varepsilon_{n_{2}} r_2- (\varepsilon_j+\varepsilon_{n_{2}}) m_{2}  =\delta_{i+1}\delta_{j+1}\varepsilon_{n_{2}} r_2 +2  \delta_{j+n_{2}+1} \varepsilon_{j+1} m_{2}.
\end{equation}
Using the relations $(\varepsilon_{i}-1)=-2\delta_{i}$ and $\delta_{i+1}(1+\varepsilon_{n_{2}+i})=2\delta_{i+1}\delta_{n_{2}+1}$, as well as Lemma~\ref{lem:algebra_tau_2}(\ref{eq:algebra_i}),~(\ref{eq:a1a2}) and~(\ref{eq:defm1n1}), we have:
\begin{equation}\label{eq:defa1a2}
\text{$a_2=-4s_1-2i$ and $a_{1}=-2 (\delta_{n_{2}+1} \delta_{i+1}\delta_{j+1}r_1+\delta_i \varepsilon_{n_{2}} m_{2})$}
\end{equation}
Let $g=m_{1}+\varepsilon_i m_{2}+\varepsilon_{n_{2}+i} a_1$. Then by~(\ref{eq:yexp}) and Proposition~\ref{prop:p2_k2}, there exist $x,y \in \gsigma$ such that:
\begin{align}
y =& v^{-b_{2}} u^{a_{1}-b_{1}} v^{a_{2}} x \theta(m_{1},\delta_{n_{1}})(u^{b_{1}}v^{b_{2}}y)\ldotp\notag\\
&\theta(m_{1}+(-1)^{\delta_{n_{1}}}m_{2},n_{1}+n_{2})((Bu^{-1})^{a_{1}} B^{-a_{1}} \theta(a_{1},0)((Bv^{2})^{-a_{2}/2}) \theta(a_{1},0)(\rho(x) B^{\delta_{n_{1}}}))\notag\\
=& v^{-b_2} u^{a_1-b_1} v^{a_2} x B^{m_{1}-\delta_i} u^{\varepsilon_i b_1} (B^{\delta_i} v u^{-2m_{1}})^{b_2} B^{-m_{1}+\delta_i}\theta(m_{1},\delta_i)(y)\, 
B^{m_{1}+\varepsilon_i m_{2}-\delta_{n_{2}+i}} (B^{\varepsilon_{n_{2}+i}}  u^{\varepsilon_{n_{2}+i+1}} )^{a_1}\ldotp\notag\\
& (B^{\varepsilon_{n_{2}+i}}(B^{\delta_{n_{2}+i}} v u^{-2g})^2)^{-a_2/2} B^{-g+\delta_{n_{2}+i}} \theta(g, \delta_{n_{2}+i})(\rho(x)) B^{\delta_i \varepsilon_{n_{2}+i}}\notag\\
=& v^{a_2-b_2} [v^{-a_2},u^{a_1-b_1}] u^{a_1-b_1}  x B^{m_{1}-\delta_i}  u^{-a_1+b_1} \left(u^{a_1-b_1 +\varepsilon_i b_1} v^{b_2}  
u^{-a_1 \varepsilon_{n_{2}+i}} v^{-b_2} \right) v^{-a_2+b_2} . v^{a_2} u^{a_1 \varepsilon_{n_{2}+i}}\ldotp\notag\\
& \left(v^{-b_2}(B^{\delta_i} v u^{-2m_{1}})^{b_2} \right) B^{-m_{1}+\delta_i}\theta(m_{1},\delta_i)(y) B^{m_{1}+\varepsilon_i m_{2}-\delta_{n_{2}+i}}  
u^{-a_1\varepsilon_{n_{2}+i}} \left(u^{a_1\varepsilon_{n_{2}+i}} (B^{\varepsilon_{n_{2}+i}}  u^{\varepsilon_{n_{2}+i+1}} )^{a_1} \right)\ldotp\notag\\
& v^{-a_2} \left(v^{a_2} (B^{\varepsilon_{n_{2}+i}}(B^{\delta_{n_{2}+i}} v u^{-2g})^2)^{-a_2/2} \right) B^{-g+ \delta_{n_{2}+i}} \theta(g, \delta_{n_{2}+i})(\rho(x)) B^{\delta_i \varepsilon_{n_{2}+i}}.\label{eq:calcy}
\end{align}

It remains to identify the four bracketed terms of~(\ref{eq:calcy}) with elements of $\gsigma$. First, by~(\ref{eq:defTIOJ}) and~(\ref{eq:defQ}) we have: 
\begin{align}
O_{s_2-n_{2},2 \delta_i m_{2}-2\delta_{i+1}\delta_{n_{2}+1}r_1}^{-\delta_{j+1}} Q_{-2\delta_i m_{2},s_2-n_{2}}^{\delta_j} &= \begin{cases}
[u^{2 \delta_i m_{2}-2\delta_{i+1}\delta_{n_{2}+1}r_1}, v^{2(s_2-n_{2})}]
& \text{if $j=0$}\\
u^{-2 \delta_i m_{2}} v^{2(s_2-n_{2})+1} u^{-2 \delta_i m_{2}} v^{-2(s_2-n_{2})-1}
& \text{if $j=1$}
\end{cases}\notag\\
&=u^{-2\delta_{i+1}\delta_{j+1}\delta_{n_{2}+1}r_1+2 \delta_i \varepsilon_{j} m_{2}} v^{b_{2}} u^{2\delta_{i+1}\delta_{j+1}\delta_{n_{2}+1}r_1-2 \delta_i  m_{2}} v^{-b_{2}}\label{eq:calcOQ}
\end{align}
using~(\ref{eq:defb2}). Now by~(\ref{eq:defb1}) and (\ref{eq:defa1a2}):
\begin{align}
-a_{1} \varepsilon_{n_{2}+i} &=2\varepsilon_{n_{2}+i}(\delta_{n_{2}+1} \delta_{i+1}\delta_{j+1}r_1+\delta_i \varepsilon_{n_{2}} m_{2})= 2(\delta_{n_{2}+1} \delta_{i+1}\delta_{j+1}r_1-\delta_i m_{2})\label{eq:calcOQi}\;\text{and}\\
a_{1}+(\varepsilon_i-1)b_{1}&= -2 (\delta_{n_{2}+1} \delta_{i+1}\delta_{j+1}r_1+\delta_i \varepsilon_{n_{2}} m_{2})-2\delta_{i} (\delta_{i+1}\delta_{j+1} \varepsilon_{n_{2}} r_2 +2  \delta_{j+n_{2}+1} \varepsilon_{j+1} m_{2})\notag\\
&= -2\delta_{n_{2}+1} \delta_{i+1}\delta_{j+1}r_1-2\delta_{i}m_{2}(\varepsilon_{n_{2}}+2\delta_{j+n_{2}+1} \varepsilon_{j+1})\notag\\
&=-2\delta_{n_{2}+1} \delta_{i+1}\delta_{j+1}r_1+2\delta_{i}\varepsilon_{j}m_{2}.\label{eq:calcOQii}
\end{align}
It follows from~(\ref{eq:calcOQ}),~(\ref{eq:calcOQi}) and~(\ref{eq:calcOQii}) that:
\begin{equation}\label{eq:resultOQ}
O_{s_2-n_{2},2 \delta_i m_{2}-2\delta_{i+1}\delta_{n_{2}+1}r_1}^{-\delta_{j+1}} Q_{-2\delta_i m_{2},s_2-n_{2}}^{\delta_j}=u^{a_1-b_1 +\varepsilon_i b_1} v^{b_2} u^{-a_1 \varepsilon_{n_{2}+i}} v^{-b_2}.
\end{equation}
Secondly, by~(\ref{eq:defTIOJ}) and~(\ref{eq:defb2}), we have:
\begin{align}
 J_{\delta_{i+1}(n_{2}-s_2),-2m_{1}} c_{-1,0}(I_{-\delta_i b_2})&=
v^{\delta_{i+1}(\delta_{j}-b_{2})} (vu^{-2m_{1}})^{\delta_{i+1}(b_{2}-\delta_{j})} v^{-1} v^{-\delta_i b_2} (vB)^{\delta_i b_2} v\notag\\
&= v^{-\delta_{i+1}b_{2}}  v^{\delta_{i+1}\delta_{j}} (vu^{-2m_{1}})^{-\delta_{i+1}\delta_{j}} (vu^{-2m_{1}})^{\delta_{i+1}b_{2}}
 v^{-\delta_i b_2} v^{-1}  (vB)^{\delta_i b_2} v\notag\\
 &= v^{-\delta_{i+1}b_{2}}  (vu^{-2m_{1}})^{\delta_{i+1}b_{2}} \ldotp
 v^{-\delta_i b_2} (v^{-1} vBv)^{\delta_i b_2}\notag\\
 &=v^{-\delta_{i+1}b_{2}}  (vu^{-2m_{1}})^{\delta_{i+1}b_{2}} \ldotp
 v^{-\delta_i b_2} (Bv)^{\delta_i b_2}=v^{-b_2}(B^{\delta_i} v u^{-2m_{1}} )^{b_2}\label{eq:calcJI}
\end{align}
using the fact that $j=\delta_{j}$ and $v^{\delta_{i+1}\delta_{j}} (vu^{-2m_{1}})^{-\delta_{i+1}\delta_{j}}=1$ for all $i,j\in \{ 0,1\}$ (recall from~(\ref{eq:defm1n1}) that $m_{1}=\delta_{i+1}\delta_{j+1}r_1$). Thirdly, it follows directly from~(\ref{eq:defTIOJ}) that:
\begin{equation}\label{eq:calcT}
T_{a_1\varepsilon_{n_{2}+i},\delta_{n_{2}+i}} =u^{a_1\varepsilon_{n_{2}+i}}(B^{\varepsilon_{n_{2}+i}}u^{\varepsilon_{n_{2}+i+1}})^{a_1}.
\end{equation}
Finally, since $a_{2}$ is even, we have:
\begin{align}
O_{a_2/2,\delta_{n_2+i+1}} c_{0,\delta_{n_2+i+1}}(J_{a_2/2,-2g+1}) &= [v^{a_{2}}, u^{\delta_{n_2+i+1}}] u^{\delta_{n_2+i+1}} v^{a_{2}} (vu^{-2g+1})^{-a_{2}} u^{-\delta_{n_2+i+1}}\notag\\
&= v^{a_{2}} u^{\delta_{n_2+i+1}} ((vu^{-2g+1})^{2})^{-a_{2}/2} u^{-\delta_{n_2+i+1}}\notag\\
&=v^{a_{2}}  (u^{\delta_{n_2+i+1}}(vu^{-2g+1})^{2}u^{-\delta_{n_2+i+1}})^{-a_{2}/2} \notag\\
&=v^{a_2} (B^{\varepsilon_{n_2+i}}(B^{\delta_{n_2+i}} v u^{-2g})^2)^{-a_2/2}.\label{eq:calcOcJ}
\end{align}
Substituting~(\ref{eq:resultOQ})--(\ref{eq:calcOcJ}) in~(\ref{eq:calcy}) and using~(\ref{eq:defcpq}) and~(\ref{eq:defTIOJ}), we obtain:
\begin{multline}\label{eq:fuv3}
y=c_{a_2-b_2,0} (O_{2s_1+i,a_1-b_1}) c_{a_2-b_2,a_1-b_1}(x) B_{a_2-b_2,a_1-b_1}^{m_{1}-\delta_i} c_{a_2-b_2,0}(O_{s_2-n_{2},2 \delta_i m_{2} -2\delta_{i+1}\delta_{n_{2}+1}r_1}^{-\delta_{j+1}} Q_{-2\delta_i m_{2},s_2-n_{2}}^{\delta_j}) \\
c_{a_2,a_1\varepsilon_{n_{2}+i}}(J_{\delta_{i+1}(n_{2}-s_2),-2m_{1}}) c_{a_2,a_1\varepsilon_{n_{2}+i}}(c_{-1,0}(I_{-\delta_i b_2})) B_{a_2,a_1\varepsilon_{n_{2}+i}}^{-m_{1}+\delta_i}
c_{a_2,a_1\varepsilon_{n_{2}+i}}(\theta(m_{1},\delta_i)(y)) B_{a_2,a_1\varepsilon_{n_{2}+i}}^{m_{1}+\varepsilon_i m_{2}-\delta_{n_{2}+i}} \\ 
c_{a_2,0}(T_{a_1\varepsilon_{n_{2}+i},\delta_{n_{2}+i}}) 
 O_{a_2/2,\delta_{n_{2}+i+1}} c_{0,\delta_{n_{2}+i+1}}(J_{a_2/2,-2g+1}) B_{0,0}^{-g+ \delta_{n_{2}+i}} \theta(g, \delta_{n_{2}+i})(\rho(x)) B_{0,0}^{\delta_i \varepsilon_{n_{2}+i}}.
\end{multline}
By~(\ref{eq:homo_gsigmab_cpq}), for all $k,l\in \z$, we have:
\begin{align}\label{eq:coc}
(c_{a_2,a_1\varepsilon_{n_{2}+i}})_\ab ((c_{-1,0})_\ab (B_{k,l})) 
= &(c_{a_2,a_1\varepsilon_{n_{2}+i}})_\ab (B_{k-1,l})
= B_{k+a_2-1,l+\varepsilon_{k}a_1\varepsilon_{n_2+i+1}} \nonumber \\
= & (c_{a_2-1,a_1\varepsilon_{n_2+i+1}})_\ab (B_{k,l}).
\end{align}
By abuse of notation, let $x$ and $y$ denote the projection in $\gsigmab$ of the elements $x$ and $y$ respectively, and let $m=m_{2}$ and $n=n_{2}$.
Projecting~(\ref{eq:fuv3}) in $\gsigmab$ and using~(\ref{eq:defm1n1}),~(\ref{eq:defa1a2}),~(\ref{eq:coc}) and Proposition~\ref{prop:gsigma}(\ref{it:wordsc}), we obtain~(\ref{eq:lema_1}) as required.
\end{proof}

\section{Proof of Theorem~\ref{th:BORSUK_TAU_2}}\label{sec:borsuk_2}


In this section, we prove Theorem~\ref{th:BORSUK_TAU_2}. Its proof will follow from the following five propositions. 
\begin{proposition}\label{prop:tau_2_case_23I}
Suppose that 
 $\alpha_\#: \begin{cases}
(1,0) \longmapsto (0 , 2s+1) \\
(0,1) \longmapsto (0,(2z+1)w)
\end{cases}$ for some $s \in \z$ and $z,w \in \{ 0,1 \}$. Then $\alpha$ has the Borsuk-Ulam property with respect to $\tau_2$.
\end{proposition}

\begin{proposition}\label{prop:tau_2_case_14I}
Suppose that $\alpha_\#:
\begin{cases}
(1,0) \longmapsto (0,(2s+1)(1-w)) \\
(0,1) \longmapsto (0,(2z-1)w + 2)
\end{cases}
$
for some $s \in \z$ and $z,w \in \{0,1\}$.
 Then $\alpha$ does not have the Borsuk-Ulam property with respect to $\tau_2$.
\end{proposition}

\begin{proposition}\label{prop:tau_2_case_5I}
Suppose that $\alpha_\#: \begin{cases}
	 (1,0) \longmapsto (0 , 2s) \\
	 (0,1) \longmapsto (0 , 2z+1) 
	\end{cases}$ for some $s \in \z \setminus\{ 0 \}$ and $z \in \{ 0,1 \}$. Then $\alpha$ has the Borsuk-Ulam property with respect to $\tau_2$.
\end{proposition}
	
\begin{proposition}\label{prop:tau_2_case_7I}
Suppose that  $\alpha_\#: \begin{cases}
	 (1,0) \longmapsto (r_1 , 2s) \\
	 (0,1) \longmapsto (r_2 , 2z)
\end{cases}$	 for some $r_1, r_2, s \in \z$ and $z \in \{ 0,1 \}$, where $r_1 \geq 0$, satisfy one of the following conditions:
\begin{enumerate}[(i)]
\item\label{it:case_7I_1} $r_2 s \neq 0$.
\item\label{it:case_7I_2} $r_1 > 0$, $r_2$ is even and $z = 0$.
\item\label{it:case_7I_3} $r_1 = r_2 = z = 0$ and $s \neq 0$.
\end{enumerate}
Then $\alpha$ has the Borsuk-Ulam property with respect to $\tau_2$.
%
\end{proposition}

\begin{proposition}\label{prop:tau_2_case_6I}
For $n \in \z$, let $\omega_n=1$ (resp.\ $\omega_n=0$) if $n=0$ (resp.\ if  $n \neq 0$), and let $\delta_n$ as defined in Proposition~\ref{prop:p2_k2}.
Suppose that $\alpha_\#: \begin{cases}
	 (1,0) \longmapsto ((z +(1-z)\delta_{r_2})r_1, 2zs)\\
	 (0,1) \longmapsto (\omega_{zs}r_2 , 2z)
\end{cases}$
for some $r_1, r_2, s \in \z$  and $z \in \{ 0,1 \}$ such that $r_1 \geq 0$. Then $\alpha$ does not have the Borsuk-Ulam property with respect to $\tau_2$.
\end{proposition}

We will prove Propositions~\ref{prop:tau_2_case_23I}--\ref{prop:tau_2_case_6I} presently. Assuming for the moment that they hold, we first give the proof of Theorem~\ref{th:BORSUK_TAU_2}.

\begin{proof}[Proof of Theorem~\ref{th:BORSUK_TAU_2}.]
Applying~\cite[Theorem 7(b)]{GonGuaLaa1},~\cite[Proposition~2.2]{GonGuaLaa2} and Corollary~\ref{cor:reduction_tau_2}, we claim that Theorem~\ref{th:BORSUK_TAU_2} follows from 
Propositions~\ref{prop:tau_2_case_23I}--\ref{prop:tau_2_case_6I}. To see this,
let $\alpha \in [ \torus, * ; \klein,*]$ be a pointed homotopy class, and let $\alpha_\# : \pi_1(\torus)\to \pi_1(\klein)$ be the homomorphism described in Remark~\ref{rem:homotopy_pi1},
which is of one of the types~1--4 given in the statement of Proposition~\ref{prop:set_homotopy}. 
\begin{enumerate}[(a)]
\item\label{it:summarya} Suppose that $\alpha_\#$ is of Type~1, 2 or~3, and let $i\in \{0,1\}$, $s_{1}$ and $s_{2}$ be the integers that appear in the description of $\alpha_\#$ in Proposition~\ref{prop:set_homotopy}. By~\cite[Proposition~2.2]{GonGuaLaa2} and Corollary~\ref{cor:reduction_tau_2}, it suffices to 
consider the cases where $i=0$, and 
$s_2 \in \{0,1\}$. In order to apply Propositions~\ref{prop:tau_2_case_23I}--\ref{prop:tau_2_case_14I}, we may subdivide these cases as follows. 
\begin{enumerate}[(i)]
\item\label{it:summaryai} If $\alpha_{\#}$ is of Type~1 and $s_{2}=0$ (resp.\ $s_2=1$) then $\alpha$ has (resp.\ does not have) the Borsuk-Ulam property with respect to $\tau_2$ by Proposition~\ref{prop:tau_2_case_23I} (resp.\ Proposition~\ref{prop:tau_2_case_14I}).

\item\label{it:summaryaiii} If $\alpha_{\#}$ is of Type~2 then $\alpha$ has the Borsuk-Ulam property with respect to $\tau_2$ by Proposition~\ref{prop:tau_2_case_23I}.

\item\label{it:summaryaiv}  If $\alpha_{\#}$ is of Type~3 then $\alpha$ has the Borsuk-Ulam property with respect to $\tau_2$ if and only if $s_{1}\neq 0$ by Propositions~\ref{prop:tau_2_case_14I} and~\ref{prop:tau_2_case_5I}.
\end{enumerate}

\item\label{it:summaryb} Suppose that $\alpha_\#$ is of Type~4, and let $r_{1},r_{2},s_{1}$ and $s_{2}$ be the integers that appear in the description of $\alpha_\#$ in Proposition~\ref{prop:set_homotopy}, where $r_{1}\geq 0$.  By Corollary~\ref{cor:reduction_tau_2}, it suffices to consider the cases where 
$s_2 \in \{0,1\}$. In order to apply Propositions~\ref{prop:tau_2_case_7I}--\ref{prop:tau_2_case_6I}, we may subdivide these cases as follows. 

\begin{enumerate}[(i)]
\item Suppose that $s_1 r_2\neq 0$. Then 
$\alpha$ has the Borsuk-Ulam property with respect to $\tau_2$ by Proposition~\ref{prop:tau_2_case_7I}(\ref{it:case_7I_1}).

\item Suppose that $s_1\neq 0$ and $r_2=0$. If $s_{2}=1$ then $\alpha$ does not have the Borsuk-Ulam property with respect to $\tau_2$ by Proposition~\ref{prop:tau_2_case_6I}. So suppose that $s_{2}=0$. 
If $r_{1}>0$ (resp.\ $r_{1}=0$) then $\alpha$ has the Borsuk-Ulam property with respect to $\tau_2$ by Proposition~\ref{prop:tau_2_case_7I}(\ref{it:case_7I_2}) (resp.\ Proposition~\ref{prop:tau_2_case_7I}(\ref{it:case_7I_3})). 

\item Finally, suppose that $s_{1}=0$. If $s_{2}=0$, $r_{2}$ is even and $r_{1}>0$ then $\alpha$ has the Borsuk-Ulam property with respect to $\tau_2$ by Proposition~\ref{prop:tau_2_case_7I}(\ref{it:case_7I_2}). In the remaining cases (either $s_{2}=1$, or $s_{2}=0$, and either $r_{2}$ is odd, or else $r_{2}$ is even and $r_{1}=0$), $\alpha$ does not have the Borsuk-Ulam property with respect to $\tau_2$ by Proposition~\ref{prop:tau_2_case_6I}.\qedhere
\end{enumerate}
\end{enumerate}
\end{proof}

%
	

The rest of this section is devoted to proving Propositions~\ref{prop:tau_2_case_23I}--\ref{prop:tau_2_case_6I}.

\begin{proof}[Proof of Proposition~\ref{prop:tau_2_case_23I}.]
Suppose on the contrary that $\alpha$ does not have the Borsuk-Ulam property with respect to $\tau_2$. Applying Lemma~\ref{lem:lemma_1} with $s_1 = s$, $i =1$, $s_2= zw$ and $j = w$, there exist $m,n \in \z$ and $x,y \in \gsigmab$ that satisfy the following equation in~$\gsigmab$:
\begin{multline}\label{eq:prop23}
(c_{2n-(2z+1)w-4s-2,2m\varepsilon_{w}\delta_{n+w}})_\ab (x)+\theta(m,\delta_{n+1})_\ab (\rho_\ab (x))+(c_{-4s-2,2m})_\ab (\theta(0,1)_\ab(y))-y\\
+(c_{-4s-2,0})_\ab(\widetilde{T}_{2m,\delta_{n+1}})
+(c_{2n-(2z+1)w-4s-2,0})_\ab (\widetilde{O}_{2s+1,2m\varepsilon_{w}\delta_{n+w}}-\delta_{w+1} \widetilde{O}_{zw-n,2m}+\delta_w \widetilde{Q}_{-2m,zw-n}) \\
+(c_{-4s-3,-2m})_\ab(\widetilde{I}_{2n-(2z+1)w}) 
+(c_{0,\delta_{n}})_\ab (\widetilde{J}_{-2s -1,1-2m})  +\widetilde{O}_{-2s-1,\delta_{n}}\\ +(\delta_{n}-m)(B_{0,0}+B_{-4s-2,2m})-B_{2n-(2z+1)w-4s-2,2m\varepsilon_{w}\delta_{n+w}}= 0,
\end{multline}
where we have also used the fact that $\varepsilon_{n}+\delta_{w+n+1} \varepsilon_{w+1} =-\varepsilon_{w}\delta_{n+w}$. Let $\map{\xi}{\gsigmab}[\ztwo=\{ \overline{0} , \overline{1} \}]$ be the homomorphism defined on the basis $\{ B_{k,l} \}_{k,l \in \z}$ of $\gsigmab$ by $\xi(B_{k,l})=\delta_{w}\overline{1}+\delta_{w+1}\overline{k}$ for all $k,l \in \z$. To prove the result, we will arrive at a contradiction by computing the image by $\xi$ of~(\ref{eq:prop23}). 
Clearly $\xi(B_{k,l})= \xi(B_{k,0})$ for all $k,l\in \z$. Using Proposition~\ref{prop:gsigma}, we may check that $\xi \circ \theta(m,n)_\ab=\xi \circ \rho_\ab= \xi$ for all $m,n\in \z$, and that $\xi \circ (c_{p,q})_\ab=\xi$ if 
$w$ is odd or $p$ is even, from which it follows that the image of the first line of~(\ref{eq:prop23}) by $\xi$ is equal to $\overline{0}$. 
Note that if $k=0$ or $l=0$ then $\xi(\widetilde{O}_{k,l})=\xi(0)=\overline{0}$, while if $k,l\in \z \setminus \{0\}$ then:
\begin{align*}
\xi(\widetilde{O}_{k,l})&=  \sum_{i=1}^{\lvert k \rvert} \sum_{j=1}^{\lvert l\rvert} \xi\bigl(
B_{\sigma_k (2i-1 ),-\sigma_l j+(\sigma_l -1 )/2} 
- B_{ \sigma_k (2i-1) -1 , \sigma_l j-(1+\sigma_l)/2}
\bigr)=  \lvert k \rvert \lvert l \rvert \left( (\delta_{w} \overline{1}+\delta_{w+1} \overline{1}) +\delta_{w} \overline{1}\right)\\
&=\lvert k \rvert \lvert l \rvert \delta_{w+1} \overline{1}.
\end{align*}
In particular, if $w$ is odd, or if either $k$ or $l$ is even then $\xi(\widetilde{O}_{k,l})=\overline{0}$. Taking the image of~(\ref{eq:prop23}) and using the above computations involving compositions with $\xi$ and the fact that $\delta_{w+1}\overline{w}=\overline{0}$, we obtain:
\begin{equation}\label{eq:prop23a}
\xi(\widetilde{T}_{2m,\delta_{n+1}})+ \delta_{w}\xi(\widetilde{Q}_{-2m,zw-n}) +\xi\circ(c_{-4s-3,-2m})_\ab(\widetilde{I}_{2n-(2z+1)w}) 
+\xi (\widetilde{J}_{-2s -1,1-2m})  +\overline{\delta_{w+1}\delta_{n}}+\overline{\delta_{w}}= \overline{0}.
\end{equation}
If $m = 0$ then $\xi(\widetilde{T}_{2m,\delta_{n+1}}) = \xi(\widetilde{Q}_{-2m,zw-n}) = \xi(0) = \overline{0}$, while if $m\neq 0$ then:
%
\begin{align}
\xi(\widetilde{T}_{2m,\delta_{n+1}})&= \sum_{i=1}^{\lvert 2m \rvert} \xi(B_{0,\sigma_{m}(i+(\sigma_{m}(1-2\delta_{n+1})-1)/2)})=\sum_{i=1}^{\lvert 2m \rvert} \xi(B_{0,0})=\sum_{i=1}^{\lvert 2m \rvert} \overline{\delta_{w}}=\overline{0}\label{eq:xit2m}\\
\xi(\widetilde{Q}_{-2m,zw-n}) &=\xi(\widetilde{O}_{zw-n,-2m})+\sum_{i=1}^{\lvert 2m\rvert} \xi(B_{2(zw-n),\sigma_{-m} i-(1+\sigma_{-m})/2})=
\sum_{i=1}^{\lvert 2m\rvert} \overline{\delta_{w}}=\overline{0}.\label{eq:qm2m}
\end{align}
Now:
\begin{align}
\xi (\widetilde{J}_{-2s -1,1-2m})&= \sum_{i=1}^{\lvert 2s +1\rvert} \sum_{j=1}^{\lvert 2m-1 \rvert} \xi(B_{-\sigma_{2s+1}(2i-1), \sigma_{1-2m}(j-(1+\sigma_{1-2m})/2)})=\sum_{i=1}^{\lvert 2s +1\rvert} \sum_{j=1}^{\lvert 2m-1 \rvert} \xi(B_{-\sigma_{2s+1}(2i-1), 0})\notag\\
&=\sum_{i=1}^{\lvert 2s +1\rvert} \sum_{j=1}^{\lvert 2m-1 \rvert}  (\overline{\delta_{w}}+\overline{\delta_{w+1}})= \sum_{i=1}^{\lvert 2s +1\rvert} \sum_{j=1}^{\lvert 2m-1 \rvert} \overline{1}=\overline{1}.\label{eq:j2s1}
\end{align}
It follows from~(\ref{eq:prop23a})--(\ref{eq:j2s1}) that:
\begin{equation}\label{eq:prop23b}
\xi\circ(c_{-4s-3,-2m})_\ab(\widetilde{I}_{2n-(2z+1)w}) 
+\overline{1}  +\overline{\delta_{w+1}\delta_{n}}+\overline{\delta_{w}}= \overline{0}.
\end{equation}
Assume first that $w=0$. If $n=0$ then~(\ref{eq:prop23b}) gives rise to a contradiction. So suppose that $n \neq 0$. Then: 
\begin{align*}
\xi\circ(c_{-4s-3,-2m})_\ab(\widetilde{I}_{2n-(2z+1)w})&= \xi\circ(c_{-4s-3,-2m})_\ab(\widetilde{I}_{2n})=\xi\circ(c_{-4s-3,-2m})_\ab \Biggl(\sum_{i=1}^{2\lvert n\rvert} B_{\sigma_{n}  i+(1-\sigma_{n} )/2 ,0}\Biggr)\\
&= \sum_{i=1}^{2\lvert n\rvert} \xi(B_{\sigma_{n}  i+(1-\sigma_{n} )/2 -4s-3,0})=\overline{\delta_{n}}.
\end{align*}
The last equality follows from the fact that as $i$ varies between $1$ and $2\lvert n\rvert$, $\sigma_{n}  i+(1-\sigma_{n} )/2 -4s-3$ runs over $2\lvert n\rvert$ consecutive integers of which exactly $\lvert n\rvert$ are odd. Equation~(\ref{eq:prop23b}) then yields a contradiction. Finally if $w=1$ then:
\begin{align*}
\xi\circ(c_{-4s-3,-2m})_\ab(\widetilde{I}_{2n-(2z+1)w}) &=\xi\circ(c_{-4s-3,-2m})_\ab(\widetilde{I}_{2n-2z-1})\\
&=  \sum_{i=1}^{\lvert 2n-2z-1\rvert} \xi(B_{\sigma_{2n-2z-1}  i+(1-\sigma_{2n-2z-1})/2-4s-3 ,0})=\overline{1},
\end{align*}
using the definition of $\xi$. Once more, (\ref{eq:prop23b}) gives rises to a contradiction. We conclude that $\alpha$ has the Borsuk-Ulam property with respect to $\tau_2$.
\end{proof}

\begin{proof}[Proof of Proposition~\ref{prop:tau_2_case_14I}]
Let $a = (v^{(4s+2)(w-1)}x;0,(2s+1)(1-w))$ and $b=(v^w; 0, w(z-1)+1)$, where $x = (v^{2s+2} (Bv^2)^{-s-1})^{1-w} \in \gsigma$. Note that $a,b\in P_2(\klein)$. To prove the result, it suffices to show that these elements satisfy  Lemma~\ref{lem:algebra_tau_2}(\ref{eq:algebra_i})--(\ref{eq:algebra_iii}). Clearly $(p_1)_\# (a)= \alpha_\# (1,0)$, so Lemma~\ref{lem:algebra_tau_2}(\ref{eq:algebra_ii}) holds. With the notation of~(\ref{eq:normformab}), we have $a_{1}=m_{1}=b_{1}=m_{2}=0$, $y=\id$ and $a_{2}$ is even. Taking $a=b$ in~(\ref{eq:blsiga}), we may check that  $(p_1)_\# (bl_{\sigma}(b))=(0,b_{2}+2n_{2})$, and this may be seen to be equal to $\alpha_\# (0,1)$, so Lemma~\ref{lem:algebra_tau_2}(\ref{eq:algebra_iii}) is satisfied too. It remains to show that Lemma~\ref{lem:algebra_tau_2}(\ref{eq:algebra_i}) holds. As in~\cite[p.~534]{GonGuaLaa2}, let $\map{p_{F}}{P_2(\klein)}[F(u,v)]$ be defined by $p_F(w; m,n)=w$ for all $w\in F(u,v)$ and $(m,n)\in \z \rtimes \z$. Then $w=(p_{F}(w); (p_{1})_{\#}(w))$, and to prove that $abl_\sigma(a)=b$, it thus suffices to show that $(p_1)_\# (abl_{\sigma}(a))= (p_1)_\# (b)$ and that $p_{F}(abl_{\sigma}(a))= p_{F}(b)$. By~(\ref{eq:a_b_lsigma_a1}), we have:
\begin{equation}\label{eq:ablsigared}
abl_\sigma(a)= (v^{a_2} x \theta(0,\delta_{n_1})(v^{b_2}) \theta(0,\delta_{n_1+n_2})((uv)^{-a_2}  \rho(x)B^{\delta_{n_{1}}});  0, 2n_1+n_2+a_2).
\end{equation}
One may check easily that $(p_1)_\# (abl_{\sigma}(a))= (p_1)_\# (b)$. Hence it remains to show that $p_{F}(abl_{\sigma}(a))= p_{F}(b)$. By~(\ref{eq:ablsigared}), if $w=1$ then $x=\id$ and $p_F(abl_{\sigma}(a))=v=p_F(b)$. So suppose that $w=0$. Then by~(\ref{eq:ablsigared}) and the fact that $Bv^{2}=(uv)^{2}$, we have:
\begin{equation}\label{eq:ablsigared2}
p_F(abl_{\sigma}(a)) =v^{-4s-2} v^{2s+2}(Bv^{2})^{-s-1} (uv)^{4s+2}  \rho(x) B= v^{-2s} (uv)^{2s} \rho(x) B. 
\end{equation}
Recall that $\rho(x)=p_{F}(l_{\sigma}(x);0,0)$ by~\cite[equation~(3.2)]{GonGuaLaa2}. 
Now $x=v^{2s+2}(Bv^{2})^{-s-1}$, and using Proposition~\ref{prop:p2_k2}, we obtain:
\begin{align*}
l_{\sigma}(x;0,0)&= l_{\sigma}(v^{2s+2}(Bv^{2})^{-s-1};0,0)=((uv)^{-2s-2};0,2s+2)(B(uv)^{-2};0,2)^{-s-1}\\
&=((uv)^{-2s-2}(B(uv)^{-2})^{-s-1};0,0).
\end{align*}
So $\rho(x)=(uv)^{-2s-2}(B(uv)^{-2})^{-s-1}$. Since $(uv)^{2}=Bv^{2}$, we see from~(\ref{eq:ablsigared2}) that:
\begin{align*}
p_F(abl_{\sigma}(a))&= v^{-2s} (uv)^{-2}(B(uv)^{-2})^{-s-1} B=v^{-2s-2} B^{-1}
(Bv^{-2}B^{-1})^{-s-1}B=\id=p_F(b).
\end{align*}
It follows that $abl_{\sigma}(a)=b$, so Lemma~\ref{lem:algebra_tau_2}(\ref{eq:algebra_i}) holds, and this completes the proof of the proposition.
\end{proof}

\begin{proof}[Proof of Proposition~\ref{prop:tau_2_case_5I}]
We argue by contradiction. Suppose that $\alpha$ does not have the Borsuk-Ulam property with respect to $\tau_2$. By Lemma~\ref{lem:lemma_1}, there exist $m,n \in \z$ and $x,y \in \gsigmab$ that satisfy the following equation in~$\gsigmab$:
\begin{multline}\label{eq:tau_2_case_5}
(c_{2n-2z-4s-1,-2\delta_n m})_\ab (x)+\theta(m,\delta_n)_\ab(\rho_\ab(x))+(c_{-4s,0})_\ab(y)-y 
+ (c_{2n-2z-4s-1,0})_\ab(\widetilde{O}_{2s,-2\delta_n m}) \\
+ (c_{0,\delta_{n+1}})_\ab  (\widetilde{J}_{-2s,1-2m})+\widetilde{O}_{-2s,\delta_{n+1}}
+(m-\delta_n)(B_{-4s,0}-B_{0,0}) =0.
\end{multline}
Let $\map{\xi}{\gsigmab}[\ztwo]$ be the homomorphism defined on the basis $\{ B_{k,l} \}_{k,l \in \z}$ of $\gsigmab$ by:
\begin{equation}\label{eq:defxi3}
\xi (B_{k,l})=\begin{cases}
\overline{1} & \text{if $k \equiv 0 \bmod{4s}$ or $k \equiv 2n-2z-1 \bmod{4s}$}\\
\overline{0} & \text{otherwise,}
\end{cases}
\end{equation}
which implies that:
\begin{equation}\label{eq:bkts}
\text{$\xi(B_{k,l})=\xi(B_{k+4ts,0})$ for all $k,l,t,s\in\z$.}
\end{equation}
Now if $s>0$ and $t\in \z$ then:
\begin{align}
\bigl\{ -\sigma_{s}(2i-1)+t + 4\lvert s \rvert \; \bigl| \; i=1,\ldots, 2\lvert s \rvert\bigr\}&= \bigl\{ (2i-1)+t  \; \bigl| \; i=1,\ldots, 2\lvert s \rvert\bigr\}\notag\\
&= \bigl\{ -\sigma_{-s}(2i-1)+t  \; \bigl| \; i=1,\ldots, 2\lvert s \rvert\bigr\}.\label{eq:sigs4s}
\end{align}
It follows from~(\ref{eq:bkts}) and~(\ref{eq:sigs4s}) that for all $s,t\in \z$, $s\neq 0$:
\begin{equation}\label{eq:xibt0}
\xi\left(\sum_{i=1}^{2\lvert s \rvert} B_{t-\sigma_{s}(2i-1),0}\right)= \xi\left(\sum_{i=1}^{2\lvert s \rvert} B_{2i-1+t,0}\right).
\end{equation}
We claim that the expression on the right-hand side of~(\ref{eq:xibt0}) is equal to $\overline{1}$. 
%
%
To prove the claim, 
note that as $i$ varies between $1$ and $2\lvert s \rvert$, the index $2i-1+t$ takes successively the values $1+t, 3+t, \ldots, 4\lvert s \rvert-1+t$. In particular, if $t$ is even (resp.\ odd), $2i-1+t$ is never congruent to $0\bmod{4s}$ (resp.\ to $2n-2z-1 \bmod{4s}$), and it is congruent to $2n-2z-1 \bmod{4s}$ (resp.\ to $0\bmod{4s}$) for precisely one value of $i$, which using~(\ref{eq:defxi3}) proves the claim. It follows from~(\ref{eq:xibt0}) that for all $s,t\in \z$, $s\neq 0$:
\begin{equation}\label{eq:xi3sum1}
\xi\Biggl(\sum_{i=1}^{2\lvert s \rvert} B_{t-\sigma_{s}(2i-1),0}\Biggr)=\overline{1}.
\end{equation}
Let $p,q\in \z$. Applying~(\ref{eq:homo_gsigmab_cpq}) and~(\ref{eq:xi3sum1}), if $u=0$ then $\xi \circ (c_{p,q})_{\ab}(\widetilde{O}_{2s,u})=\overline{0}$, while if $u\neq 0$ we obtain:
\begin{equation*}
\xi \circ (c_{p,q})_{\ab}(\widetilde{O}_{2s,u})= \xi\left( \sum_{i=1}^{2\lvert s\rvert} \sum_{j=1}^{\lvert u \rvert} B_{\sigma_{2s}(2i-1)+p,0}-B_{\sigma_{2s}(2i-1)+p-1,0}\right)= \lvert u \rvert(\overline{1}+\overline{1})=\overline{0}.
\end{equation*}
In a similar manner, one sees that $\xi(\widetilde{J}_{-2s,-2m+1})=\overline{1}$. Taking the image of~(\ref{eq:tau_2_case_5}) by $\xi$, we conclude that:
\begin{equation}\label{eq:xi3red}
\xi((c_{2n-2z-4s-1,-2\delta_n m})_\ab (x)+\theta(m,\delta_n)_\ab(\rho_\ab(x)))+\xi((c_{-4s,0})_\ab(y)-y) =\overline{1}.
\end{equation}
Using~(\ref{eq:homo_gsigmab_cpq}) once more, for all $k,l\in \z$, we have:
\begin{equation*}
\xi \left((c_{-4s,0})_\ab(B_{k,l}) \right)+\xi \left(B_{k,l} \right)
= \xi \left(B_{k-4s,0} \right)+\xi \left(B_{k,0} \right)
= \xi \left(B_{k,0} \right)+\xi \left(B_{k,0} \right)=\overline{0},
\end{equation*}
from which we see that $\xi((c_{-4s,0})_\ab(y)-y)=\overline{0}$. Further, by~(\ref{eq:homo_gsigmab_theta})--(\ref{eq:homo_gsigmab_cpq}), we obtain:
\begin{equation}\label{eq:xi3bkl}
\xi \left((c_{2n-2z-4s-1,-2\delta_n m })_\ab (B_{k,l})+(\theta(m,\delta_n)_\ab \circ \rho_\ab)(B_{k,l}) \right)= \xi(B_{k-2z+2n-1,0})+\xi (B_{-k,0} ).
\end{equation}
Now:
\begin{align*}
k-2z+2n-1 \equiv 0 \bmod{4s} &\Longleftrightarrow -k \equiv -2z+2n-1 \bmod{4s}, \;\text{and}\\
k-2z+2n-1 \equiv -2z+2n-1 \bmod{4s} &\Longleftrightarrow-k \equiv 0 \bmod{4s},
\end{align*}
from which it follows using~(\ref{eq:xi3bkl}) that $\xi \left((c_{2n-2z-4s-1,-2\delta_n m })_\ab (B_{k,l})+(\theta(m,\delta_n)_\ab \circ \rho_\ab)(B_{k,l}) \right)=\overline{0}$, and hence that $\xi((c_{2n-2z-4s-1,-2\delta_n m})_\ab (x)+\theta(m,\delta_n)_\ab(\rho_\ab(x)))=\overline{0}$. Equation~(\ref{eq:xi3red}) then yields a contradiction. We thus conclude that $\alpha$ has the Borsuk-Ulam property with respect to $\tau_2$.
\end{proof}

\begin{proof}[Proof of Proposition~\ref{prop:tau_2_case_7I}] 
Suppose on the contrary that $\alpha$ does not have the Borsuk-Ulam property with respect to $\tau_2$. Applying Lemma~\ref{lem:lemma_1} with $i=j=0$, $s_{1}=s$ and $s_{2}=z$, there exist $m,n \in \z$ and $x,y \in \gsigmab$ that satisfy the following equation in~$\gsigmab$:
%
\begin{multline}\label{eq:tau_2_case_7}
\mu(x)+\nu(y)
+(c_{-4s,0})_\ab (\widetilde{T}_{-2\delta_{n+1}r_1,\delta_{n}}) + (c_{2n-2z-4s,0})_\ab(\widetilde{O}_{2s,2\delta_{n+1}(m-r_1)+\varepsilon_{n+1}r_2}  
-\widetilde{O}_{z-n,-2\delta_{n+1}r_1}) \\
+ (c_{-4s,-2\delta_{n+1}r_1})_\ab(\widetilde{J}_{n-z,-2r_1}) 
+ (c_{0,\delta_{n+1}})_\ab(\widetilde{J}_{-2s,2\varepsilon_n r_1-2m+1})+\widetilde{O}_{-2s,\delta_{n+1}}
+(\varepsilon_n r_1-m +\delta_n)B_{0,0}\\
+(m-\delta_n)B_{-4s,-2\delta_{n+1}r_1} +r_1 B_{2n-2z-4s,2\delta_{n+1}(m-r_1)+\varepsilon_{n+1}r_2}=0,
\end{multline}
where $\map{\mu,\nu}{\gsigmab}$ are the homomorphisms defined on the basis $\{B_{k,l}\}_{k,l\in \z}$ of $\gsigmab$ by:
\begin{align*}
\mu(B_{k,l}) &= (c_{2n-2z-4s,2\delta_{n+1}(m-r_1)+\varepsilon_{n+1}r_2})_\ab (B_{k,l})+\theta(m+\varepsilon_{n+1}r_1,\delta_n)_\ab(\rho_\ab(B_{k,l})) \\
\nu(B_{k,l}) &= (c_{-4s,-2\delta_{n+1}r_1})_\ab(\theta(r_1,0)_\ab(B_{k,l}))-B_{k,l}.
\end{align*}
Using Proposition~\ref{prop:gsigma}, one may check that:
\begin{align*}
\mu(B_{k,l}) &= B_{k-2z+2n-4s,l +\varepsilon_{k}(2 \delta_{n+1}(m-r_1)+\varepsilon_{n+1}r_2)}+\varepsilon_{k+n}B_{-k,\varepsilon_{k+n+1}l-2\delta_k(m+\varepsilon_{n+1}r_1)} \\
\nu(B_{k,l}) &= B_{k-4s,l-2\delta_{n+k+1}r_1}-B_{k,l}.
\end{align*}
In what follows, we analyse in turn each of the conditions~(\ref{it:case_7I_1})--(\ref{it:case_7I_3}) of the statement of the proposition, and in each case, we will reach a contradiction.

\begin{enumerate}[(i)]
\item Suppose that $r_2  s \neq 0$, and let $\map{\xi_1}{\gsigmab}[\z]$ be the homomorphism defined on the basis $\{ B_{k,l} \}_{k,l \in \z}$ of $\gsigmab$ by $\xi_1 (B_{k,l})=\delta_{k+n}$ for all $k,l \in \z$. Then $\xi_1(B_{k,l})=\xi_1(B_{k+2t,0})$ for all $k,l,t \in \z$, from which we obtain $\xi_1 \circ \mu (B_{k,l})=\delta_{k+n} +\varepsilon_{k+n}\delta_{k+n}=0$ and
$\xi_1 \circ \nu (B_{k,l})=0$. Using~Proposition~\ref{prop:gsigma}, we see that $\xi_1(\widetilde{O}_{k,l} )=kl (\delta_{n+1}-\delta_n)=\varepsilon_n kl$, 
$\xi_1(\widetilde{J}_{k,l})=-\delta_{n+1}kl$ and $\xi_1(\widetilde{T}_{k,r}) =\delta_n k$. Taking the image of~(\ref{eq:tau_2_case_7}) by~$\xi_1$, and making use of these facts, it follows that:
\begin{multline}\label{eq:multcondi}
- 2 \delta_n \delta_{n+1}r_1+\varepsilon_n(2s(2\delta_{n+1}(m-r_1)+\varepsilon_{n+1}r_2)+ 2(z-n)\delta_{n+1}r_1-2s\delta_{n+1}) \\
-\delta_{n+1}(-2s(-2m+2\varepsilon_n r_1+1)-2(n-z)r_1))+ \delta_n(\varepsilon_n r_1+r_1) =0.
\end{multline}
Applying the equalities $\delta_{n+1}\varepsilon_n=\delta_{n+1}$, $\delta_n \delta_{n+1}=0$ and $\delta_n(1+\varepsilon_n)=0$ to~(\ref{eq:multcondi}), we obtain $-2 r_2 s=0$, which contradicts the hypothesis. 

\item Suppose that $r_1 > 0$, $r_2$ is even and $z = 0$, and let $\map{\xi_2}{\gsigmab}[\ztwo]$ be the homomorphism defined on the basis $\{ B_{k,l} \}_{k,l \in \z}$ of $\gsigmab$ by:
\begin{equation*}
\xi_2 (B_{k,l})=\begin{cases}
\overline{k+n+1} & \text{if  $l \equiv \varepsilon_n m-r_2/2 \bmod{2 \lvert r_1\rvert}$}\\
\overline{0} & \text{otherwise.}
\end{cases}
\end{equation*}
So $\xi_2 (B_{k,l})=\xi_2 (B_{k+2t,l+2u r_1})$ for all $k,l,t,u \in \z$, and $\xi_2 (B_{k,l})=\overline{0}$ if $k+n$ is odd. Also, by Proposition~\ref{prop:gsigma} we have:
\begin{equation}\label{eq:cpq}
\text{$\xi_2 \circ (c_{p,q})_\ab=\xi_2 \circ (c_{p+2t,q+2ur_1})_\ab$ for all $p,q,t,u \in \z$.}
\end{equation}
To analyse the image by $\xi_{2}$ of the terms of~(\ref{eq:tau_2_case_7}), note that if $k\in \z$ is even, $j\in \z$, 
and $\rho_{1},\rho_{2}\in \z$ then:
\begin{equation*}
\xi_{2}\Biggl(\sum_{i=1}^{\lvert k \rvert} B_{\rho_{1} \sigma_{k}(2i-1)+\rho_{2},j}\Biggr)= \sum_{i=1}^{\lvert k \rvert} \xi_{2}(B_{\rho_{2}-\rho_{1} \sigma_{k},j})= \overline{k} \xi_{2}(B_{\rho_{2}-\rho_{1} \sigma_{k},j})=\overline{0}. 
\end{equation*}
In particular, if $k$ is even and $l\in \z$ then it follows from Proposition~\ref{prop:gsigma} that:
\begin{equation}\label{eq:jkl0}
\xi_{2}(\widetilde{J}_{k,l})=\xi_{2}(\widetilde{O}_{k,l})=\overline{0}. 
\end{equation}
Using~(\ref{eq:cpq}) and~(\ref{eq:jkl0}), we conclude that the images by $\xi_{2}$ of the terms $(c_{0,\delta_{n+1}})_\ab(\widetilde{J}_{-2s,2\varepsilon_n r_1-2m+1})$, $\widetilde{O}_{-2s,\delta_{n+1}}$ and $(c_{2n-4s,0})_\ab(\widetilde{O}_{2s,2\delta_{n+1}(m-r_1)+\varepsilon_{n+1}r_2} )$ of~(\ref{eq:tau_2_case_7}) are all equal to $\overline{0}$. Similarly, we have $\xi_2\circ (c_{2n-4s,0})_\ab(\widetilde{O}_{-n,-2\delta_{n+1}r_1})=\overline{0}$ if $n$ is even, while if $n$ is odd, $\xi_2\circ (c_{2n-4s,0})_\ab(\widetilde{O}_{-n,-2\delta_{n+1}r_1})=\overline{0}$ by Proposition~\ref{prop:gsigma}. It follows that $\xi_2\circ (c_{2n-4s,0})_\ab(\widetilde{O}_{-n,-2\delta_{n+1}r_1})=\overline{0}$ for all $n$. Thus the image of~(\ref{eq:tau_2_case_7}) by $\xi_{2}$ yields:
\begin{equation}\label{eq:tau_2_case_7a}
\xi_{2}(\mu(x)+\nu(y)
+ \widetilde{T}_{-2\delta_{n+1}r_1,\delta_{n}}+ \widetilde{J}_{n,-2r_1} 
+\chi)=\overline{0},
\end{equation}
where $\chi=(\varepsilon_n r_1-m +\delta_n)B_{0,0}+(m-\delta_n)B_{-4s,-2\delta_{n+1}r_1}+r_1 B_{2n-4s,2\delta_{n+1}(m-r_1)+\varepsilon_{n+1}r_2}$. We now compute each of the terms of~(\ref{eq:tau_2_case_7a}).
\begin{enumerate}[(a)]
\item\label{it:tr1} Let us show that $\xi_{2}(\widetilde{T}_{-2\delta_{n+1}r_1,\delta_{n}})=\overline{n+1}$ for all $n\in \z$. To see this, if $n$ is odd then $\widetilde{T}_{-2\delta_{n+1}r_1,\delta_{n}}= \widetilde{T}_{0,1}=0$ by Proposition~\ref{prop:gsigma}, and so $\xi_{2}(\widetilde{T}_{-2\delta_{n+1}r_1,\delta_{n}})=\overline{0}=\overline{n+1}$. So suppose that $n$ is even. Since $r_{1}>0$, we have:
\begin{equation*}
\xi_{2}(\widetilde{T}_{-2\delta_{n+1}r_1,\delta_{n}})=\xi_{2}(\widetilde{T}_{-2r_1,0})=\sum_{i=1}^{2r_{1}} \xi_{2}(B_{0,1-i}). 
\end{equation*}
Since the set $\{ 1-i \, | \, 1\leq i\leq 2r_{1}\}$ contains precisely one element that is congruent to $m-r_2/2$ modulo $2r_{1}$, it follows that $\xi_{2}(\widetilde{T}_{-2\delta_{n+1}r_1,\delta_{n}})=\overline{n+1}$, which proves the result.

\item We claim that $\xi_{2}(\widetilde{J}_{n,-2r_1})=\overline{n}$ for all $n\in \z$. To see this, if $n$ is even then $\xi_{2}(\widetilde{J}_{n,-2r_1})=\overline{0}=\overline{n}$ by~(\ref{eq:jkl0}). So suppose that $n$ is odd. Then $n\neq 0$, and since $r_{1}>0$, we have:
\begin{equation*}
\xi_{2}(\widetilde{J}_{n,-2r_1})= \sum_{i=1}^{\lvert n \rvert} \sum_{j=1}^{2 r_1} \xi_2 (B_{\sigma_{n}(2i-1), -j})= \sum_{i=1}^{\lvert n \rvert} \sum_{j=1}^{2 r_1} \xi_2 (B_{-\sigma_{n}, -j})= \sum_{j=1}^{2 r_1} \xi_2 (B_{-\sigma_{n}, -j}).
\end{equation*}
As in case~(\ref{it:tr1}), it follows that $\xi_{2}(\widetilde{J}_{n,-2r_1})=\overline{-\sigma_{n}+n+1}=\overline{n}$, which proves the claim.

\item Let us show that the homomorphisms $\xi_{2}\circ \mu$ and $\xi_{2}\circ \nu$ are both zero. It suffices to prove that they are zero on the elements of the basis $\{ B_{k,l} \}_{k,l\in \z}$ of $\gsigmab$. Let $k,l\in \z$. Then:
\begin{equation}\label{eq:tau_2_case_7_aux1a}
\xi_2 \circ \mu (B_{k,l})= \xi_{2}(B_{k,l+2\delta_{n+1}\varepsilon_k m+\varepsilon_{k+n+1} r_2})+ \xi_2 (B_{k,\varepsilon_{n+k+1}l-2\delta_k m}).
\end{equation}
If $k+n$ is odd then it follows from~(\ref{eq:tau_2_case_7_aux1a}) and the definition of $\xi_{2}$ that $\xi_2 \circ \mu (B_{k,l}) =\overline{0}$. So suppose that $k+n$ is even. Then $\varepsilon_k=\varepsilon_n$ and $\varepsilon_{n}(1- 2\delta_{n+1})=-1=-(\varepsilon_n+2\delta_k)$. 
Hence:
\begin{align*}
l+2\delta_{n+1}\varepsilon_k m-r_2 \equiv \varepsilon_n m -r_2/2 \bmod{2r_1}  &\Longleftrightarrow 
l  \equiv \varepsilon_n(1-2\delta_{n+1})m+r_2/2 \bmod{2r_1}\\
&\Longleftrightarrow
-l  \equiv (\varepsilon_n+2\delta_k)m-r_2/2 \bmod{2r_1}\\
&\Longleftrightarrow
-l-2\delta_{k}m  \equiv \varepsilon_n m -r_2/2 \bmod{2r_1}.
\end{align*}
So the terms on the right hand-side of~(\ref{eq:tau_2_case_7_aux1a}) take the same value in $\ztwo$, and thus $\xi_2 \circ \mu (B_{k,l})=\overline{0}$. Since $\xi_{2}(B_{k-4s,l-2\delta_{n+k+1}r_1})=\xi_{2}(B_{k,l})$, it follows that $\xi_{2}\circ \nu(B_{k,l})= \xi_{2}(B_{k,l})+\xi_{2}(B_{k,l})=\overline{0}$, which proves the result.

\item\label{it:tr4} Using the definition of $\xi_{2}$ and the previous calculation, we have:
\begin{align*}
\xi_2(\chi)&=\overline{(-m +\varepsilon_n r_1 +\delta_n)}\xi_2(B_{0,0})+ \overline{(-\delta_n+m)}\xi_{2}(B_{0,0})+\overline{r_{1}}\xi_{2}(B_{2n-4s,2\delta_{n+1}(m-r_1)+\varepsilon_{n+1}r_2})\\
&=\overline{r_1}\xi_2(\varepsilon_n B_{0,0})+\overline{r_{1}}\xi_{2}(B_{2n-4s,2\delta_{n+1}(m-r_1)+\varepsilon_{n+1}r_2})\\
&= \overline{r_1}\xi_2\bigl(\theta(m+\varepsilon_{n+1}r_1,\delta_n)_\ab(\rho_\ab(B_{0,0}))+ (c_{2n-4s,2\delta_{n+1}(m-r_1)+\varepsilon_{n+1}r_2})_{\ab}(B_{0,0})\bigr)\\
&=\overline{r_1}\xi_2\circ \mu(B_{0,0})=\overline{0}.
\end{align*}
\end{enumerate}
Substituting the results of~(\ref{it:tr1})--(\ref{it:tr4}) in~(\ref{eq:tau_2_case_7a})yields a contradiction in this case.

\item Finally, suppose that $r_1=r_2=z=0$ and $s \neq 0$. Let $\map{\xi_3}{\gsigmab}[\ztwo]$ be the homomorphism defined on the basis $\{ B_{k,l} \}_{k,l \in \z}$ of $\gsigmab$ by:
\begin{equation*}
\xi_3 (B_{k,l})=\begin{cases}
\overline{1} & \text{if  $k \equiv n \bmod{4\lvert s \rvert}$}\\
\overline{0} & \text{otherwise.}
\end{cases}
\end{equation*}
Then $\xi_3(B_{k+4ts,l})=\xi_3(B_{k,0})$ for all $k,l,t\in\z$. It follows from this equality that $\xi_3\circ\nu(y)=\overline{0}$, that $\xi_{3}\circ (c_{0,\delta_{n+1}})_\ab(\widetilde{J}_{-2s,1-2m})=\xi_{3}(\widetilde{J}_{-2s,1-2m})$, and that $\xi_{3}((\delta_n-m)B_{0,0}+(m-\delta_n)B_{-4s,0})=\overline{0}$. Further, since $k + 2n \equiv n \bmod{4 \lvert s \rvert}$ if and only if $-k  \equiv n \bmod{4 \lvert s \rvert}$, we see from the definition of $\mu$ that $\xi_3(\mu(x)) = \overline{0}$. Applying $\xi_3$ to~(\ref{eq:tau_2_case_7}) and making use of Proposition~\ref{prop:gsigma}, we obtain:
\begin{equation}\label{eq:tau_2_case_7b}
\xi_3((c_{2n-4s,0})_\ab(\widetilde{O}_{2s,2\delta_{n+1}m})) + \xi_3(\widetilde{J}_{-2s,1-2m}) + \xi_3(\widetilde{O}_{-2s,\delta_{n+1}}) = \overline{0}.
\end{equation}
It remains to compute each of the terms of~(\ref{eq:tau_2_case_7b}).
\begin{enumerate}[(a)]
%
\item\label{it:xi3a} Let us show that $\xi_3((c_{2n-4s,0})_\ab(\widetilde{O}_{2s,2\delta_{n+1}m}))= \overline{0}$. If $n$ is odd or $m=0$, this is clearly the case.
So suppose that $n$ is even and $m \neq 0$. Then by Proposition~\ref{prop:gsigma} we have:
\begin{align*}
\xi_3((c_{2n-4s,0})_\ab(\widetilde{O}_{2s,2\delta_{n+1}m}))
& = \displaystyle \sum_{i=1}^{2\lvert s \lvert } \sum_{j=1}^{2 \lvert m\rvert} \left( \xi_3 \left( B_{\sigma_s(2i-1)+2n-4s,0} - B_{\sigma_s(2i-1)+2n-4s-1,0}\right) \right) \\
& = 2 \lvert m \rvert \displaystyle \sum_{i=1}^{2\lvert s \lvert }  \left( \xi_3 \left( B_{\sigma_s(2i-1)+2n-4s,0} - B_{\sigma_s(2i-1)+2n-4s-1,0}\right) \right)  = \overline{0}.
\end{align*}
\item We claim that $\xi_3(\widetilde{J}_{-2s,-2m+1}) = \overline{n}$. To see this, by Proposition~\ref{prop:gsigma} we have:
\begin{equation*}
\xi_3(\widetilde{J}_{-2s,1-2m})
= \displaystyle \sum_{i=1}^{2 \lvert s \rvert} \sum_{j=1}^{\lvert 2m-1 \rvert} \xi_3(B_{-\sigma_s(2i-1),0})
= \sum_{i=1}^{2 \lvert s \rvert}  \xi_3(B_{-\sigma_s(2i-1),0}).
\end{equation*}
Let $A=\{ -\sigma_s(2i-1) \, | \, 1 \leq i \leq 2\lvert s \rvert \}$. Then $A$ consists of all odd integers between $-\sigma_s$ and $-\sigma_s\lvert 4s \rvert$. So if $n$ is even (resp.\ odd) then there is no element (resp. exactly one element) of $A$ that is congruent to $n\bmod{4\lvert s \rvert}$, and the claim follows.

%
\item\label{it:xi3c} Let us show that $\xi_3(\widetilde{O}_{-2s,\delta_{n+1}})=\overline{n+1}$. If $n$ is odd the result is clear. So suppose that $n$ is even. By Proposition~\ref{prop:gsigma} we have:
\begin{equation*}
\xi_3(\widetilde{O}_{-2s,\delta_{n+1}})
= \displaystyle \sum_{i=1}^{2 \lvert s \rvert}\left( \xi_3(B_{-\sigma_s(2i-1),0}) + \xi_3(B_{-\sigma_s(2i-1)-1,0}) \right).
\end{equation*}
Let $A = \{ -\sigma_s(2i-1) \, | \, 1 \leq i \leq 2 \lvert s \rvert  \}$ and $B =  \{ -\sigma_s(2i-1)-1 \, | \, 1 \leq i \leq 2 \lvert s \rvert  \}$. Note that $A$ and $B$ are disjoint, and if $s<0$ (resp.\ $s>0$), $A \cup B$ is equal to $\{0 ,1, \ldots , 4 \lvert s \rvert - 1\}$ (resp.\ to $\{-1,-2, \ldots , -4 \lvert s \rvert\}$) So there is exactly one element of $A \cup B$ that is congruent to $n\bmod{4 \lvert s \rvert}$, and therefore $\xi_3(\widetilde{O}_{-2s,\delta_{n+1}}) = \overline{1} = \overline{n+1}$ as required. 
\end{enumerate}
We obtain a contradiction by substituting the results of~(\ref{it:xi3a})--(\ref{it:xi3c}) in~(\ref{eq:tau_2_case_7b}), and we conclude that $\alpha$ possesses the Borsuk-Ulam property with respect to~$\tau_2$.\qedhere
\end{enumerate}
\end{proof}

\begin{proof}[Proof of Proposition~\ref{prop:tau_2_case_6I}] 
To prove that $\alpha$ does not have the Borsuk-Ulam property with respect to $\tau_2$, it suffices to exhibit elements $a,b\in P_2(\klein)$ that satisfy conditions~(\ref{eq:algebra_i})--(\ref{eq:algebra_iii}) of Lemma~\ref{lem:algebra_tau_2}. 
We define these elements as follows:
\begin{enumerate}
\item\label{it:case6_a}
\begin{enumerate}
\item\label{it:case6_ai} if $z=0$ and $r_{2}$ is even, let $a=(\id;0,0)$ and $b=(\id;r_2/2,0)$.
\item\label{it:case6_aii} if $z=1$, let $a=(v^{-2s} (u^{2r_1-1} v^{-1})^{2s} B^{-r_1}; r_1 , 2s)$ and $b =(u^{-\omega_s r_2} B^{1-\omega_s};0,1)$
\end{enumerate}
\item\label{it:case6_b} if $z=0$ and $r_{2}$ is odd, let $a= a_1^{r_1}$ and $b=(b_1 \sigma)^{-r_2} \sigma^{-1}$, where $a_1=(u^{-2};1,0)$ and $b_1=(u^{-1};0,0)$.
\end{enumerate}
We start by considering case~(\ref{it:case6_a}). First, $(p_{1})_{\#}(a)=(0,0)$ in case~(\ref{it:case6_a})(\ref{it:case6_ai}) and $(p_{1})_{\#}(a)=(r_{1},2s)$ in case~(\ref{it:case6_a})(\ref{it:case6_aii}), and it follows that $(p_{1})_{\#}(a)=((z +(1-z)\delta_{r_2})r_1, 2zs) =\alpha_\#(1,0)$, so Lemma~\ref{lem:algebra_tau_2}(\ref{eq:algebra_ii}) holds. Secondly, by taking $a=b$ in~(\ref{eq:blsiga}), we obtain $(p_1)_\# (b l_\sigma(b))=(r_{2},0)$ in case~(\ref{it:case6_a})(\ref{it:case6_ai}) and $(p_1)_\# (b l_\sigma(b))=(\omega_{s} r_{2},2)$   in case~(\ref{it:case6_a})(\ref{it:case6_aii}), and it follows that $(p_1)_\# (b l_\sigma(b))=(\omega_{zs}r_2 , 2z)=\alpha_\#(0,1)$,  so Lemma~\ref{lem:algebra_tau_2}(\ref{eq:algebra_iii}) is satisfied. In case~(\ref{it:case6_a})(\ref{it:case6_ai}), it is clear that $abl_{\sigma}(a)=b$,  so Lemma~\ref{lem:algebra_tau_2}(\ref{eq:algebra_i}) holds, and thus $\alpha$ does not have the Borsuk-Ulam property with respect to $\tau_2$ in this case. We now show that Lemma~\ref{lem:algebra_tau_2}(\ref{eq:algebra_i}) is satisfied in case~(\ref{it:case6_a})(\ref{it:case6_aii}). Taking $a=(u^{2r_1-1} v^{-1}; 0,0)$ and $b=(\id;0,0)$ in~(\ref{eq:blsiga}) and using Proposition~\ref{prop:p2_k2} and the fact that $B=uvuv^{-1}$, we obtain:
\begin{align*}
l_{\sigma}(u^{2r_1-1} v^{-1}; 0,0)&= ((Bu^{-1})^{2r_{1}-1} B^{1-2r_{1}} \theta(2r_{1}-1,0)(uvuB); 2r_{1}-1,-1)\\
&= (uvu^{2r_{1}-1}v^{-1}u^{-1}\ldotp B^{2r_1-1} \ldotp B^{1-2r_{1}} uvu^{2(1-2r_{1})}u B^{1-2r_{1}} \ldotp B ; 2r_{1}-1,-1)\\
&=(uvu^{2-2r_{1}}B^{2-2r_{1}}; 2r_{1}-1,-1).
\end{align*}
So:
\begin{align*}
(l_{\sigma}(u^{2r_1-1} v^{-1}; 0,0))^{2s}=& ((uvu^{2-2r_1}B^{2-2r_1} ;2r_1-1,-1)(uvu^{2-2r_1}B^{2-2r_1} ;2r_1-1,-1))^{s}\\
=& (uvu^{2-2r_1}B^{2-2r_1} B^{2r_1-2}u^{-1}B^{2-2r_1} B^{2r_1-1}vu^{2-4r_1}B^{2-2r_1}\ldotp\\
& (B^{2r_1-2}u^{-1}B^{2-2r_1})^{2-2r_1}B^{2r_1-2};0,-2)^{s} \\
=& (uvu^{1-2r_1}Bvu^{2-4r_1}u^{2r_1-2};0,-2)^{s}=(uvu^{2-2r_1}vu^{1-2r_1};0,-2)^{s} \\
=& (B(vu^{1-2r_1})^2;0,-2)^{s}= ((B(vu^{1-2r_1})^2)^{s};0,-2s).
\end{align*}
Now:
\begin{align*}
\theta(r_{1},2s+1)((Bv^{2})^{s}) &= \theta(r_{1},2s+1)((uv)^{2s})=(B^{r_1-1}(u^{-1}Bvu^{-2r_{1}})B^{1-r_1})^{2s}\\
&=(B^{r_1-1}(vu^{1-2r_{1}})B^{1-r_1})^{2s})=B^{r_1-1}(vu^{1-2r_{1}})^{2s} B^{1-r_1},
\end{align*}
and
\begin{align*}
\theta(r_{1},2s+1)((B(vu^{1-2r_{1}})^{2})^{s})&=(B^{-1} (B^{r_{1}}vu^{-2r_{1}}B^{1-r_{1}}\ldotp B^{r_{1}-1} u^{2r_{1}-1} B^{1-r_{1}})^{2})^{s}\\
&= (B^{-1} (B^{r_{1}}vu^{-1}B B^{-r_{1}})^{2})^{s}= (B^{r_{1}-1} vu^{-1}Bvu^{-1} B^{1-r_{1}})^{s}\\
&= (B^{r_{1}-1} v^{2}B^{1-r_{1}})^{s}= B^{r_{1}-1} v^{2s}B^{1-r_{1}}.
\end{align*}
Hence:
\begin{align*}
l_\sigma(a) &=l_\sigma(v^{-2s};0,0) (l_\sigma (u^{2r_1-1}v^{-1};0,0))^{2s}  l_\sigma (B^{-r_1};r_1,2s) \\
&= ((uv)^{2s};0,-2s) ((B(vu^{1-2r_1})^2)^{s};0,-2s) (B^{-r_1};r_1,2s)\\
&= ((uvuv^{-1} v^2)^s;0,-2s) ((B(vu^{1-2r_1})^2)^{s};0,-2s) (B^{-r_1};r_1,2s)\\
&= ((Bv^2)^s (B(vu^{1-2r_1})^2)^sB^{-r_1};r_1,-2s).
\end{align*}
If $s=0$ then:
\begin{align*}
a b l_\sigma(a) &=(B^{-r_1};r_1,0)(u^{-r_2};0,1)(B^{-r_1};r_1,0)=(B^{-r_1}(B^{r_1} u B^{-r_1} )^{-r_2} ; r_1,1  )(B^{-r_1};r_1,0) \\
&=(u^{-r_2} B^{-r_1} B^{r_1};0,1)=(u^{-r_2};0,1)=b,
\end{align*}
while if $s\neq 0$ then:
\begin{align*}
a b l_\sigma (a) &= (v^{-2s} (u^{2r_1-1} v^{-1})^{2s} B^{-r_1}; r_1,2s)(B ; 0 , 1)((Bv^2)^s (B(vu^{1-2r_1})^2)^s B^{-r_1};r_1,-2s) \\
&= (v^{-2s} (u^{2r_1-1} v^{-1})^{2s} B^{1-r_1}; r_1 ,2s+1)((Bv^2)^s (B(vu^{1-2r_1})^2)^s B^{-r_1};r_1,-2s) \\
&= (v^{-2s} (u^{2r_1-1} v^{-1})^{2s} B^{1-r_1} \ldotp B^{r_1-1}(vu^{1-2r_{1}})^{2s} B^{1-r_1} \ldotp B^{r_{1}-1} v^{2s}B^{1-r_{1}}\ldotp B^{r_1}; 0,1)\\
&= (B;0,1)=b.
\end{align*}
So Lemma~\ref{lem:algebra_tau_2}(\ref{eq:algebra_i}) is satisfied in case~(\ref{it:case6_a})(\ref{it:case6_aii}), and hence $\alpha$ does not have the Borsuk-Ulam property with respect to $\tau_2$ in this case. 

We now consider case~(\ref{it:case6_b}). First, $(p_1)_\#(a)=(p_1)_\# (a_1)^{r_1}=(1,0)^{r_1}=(r_1,0)= \alpha_\# (1,0)$, so Lemma~\ref{lem:algebra_tau_2}(\ref{eq:algebra_ii}) holds. Further:
\begin{align*}
b l_\sigma(b) &=(b_1 \sigma)^{-r_2} \sigma^{-1} \sigma (b_1 \sigma)^{-r_2} \sigma^{-1} \sigma^{-1} 
=(b_1 \sigma b_1 \sigma)^{-r_2} B^{-1}=(b_1 l_\sigma (b_1) B)^{-r_2} B^{-1}, 
\end{align*}
and hence:
\begin{align*}
(p_1)_\# (b l_\sigma(b)) &= ((p_1)_\# (b_1) (p_1)_\# (l_\sigma (b_1)) (p_1)_\# (B))^{-r_2} (p_1)_\# (B^{-1})=(-1,0)^{r_2}=(r_2,0)=\alpha_\# (0,1),
\end{align*}
so Lemma~\ref{lem:algebra_tau_2}(\ref{eq:algebra_iii}) is satisfied. It remains to show that Lemma~\ref{lem:algebra_tau_2}(\ref{eq:algebra_i}) holds. 
By~(\ref{eq:a_b_lsigma_a1}), we have:
\begin{align*}
a_1 b_1 l_\sigma(a_1) &= (u^{-2} \theta(1,0)(u^{-1}) \theta(1,0)((Bu^{-1})^{-2} B^{2}); 0,0)=(u^{-2} \theta(1,0)(u^{-1}(Bu^{-1})^{-2} B^{2}); 0,0)\\
&= (u^{-2} \theta(1,0)(B^{-1}uB); 0,0)= (u^{-1};0,0)=b_1.
\end{align*}
So $l_\sigma(a_1)=b_1^{-1}a_1^{-1} b_1$, hence $\sigma a \sigma^{-1}= l_\sigma(a)=b_1^{-1}a^{-1} b_1$, and thus $a^{-1}=(b_1\sigma)^{-1}a (b_1\sigma)$. Since $r_{2}$ is odd, it follows that $a^{-1}=(b_1\sigma)^{-r_{2}} a (b_1\sigma)^{r_{2}}= b l_\sigma(a) b^{-1}$, which implies that $ab l_\sigma(a)=b$. So Lemma~\ref{lem:algebra_tau_2}(\ref{eq:algebra_i}) is satisfied in case~(\ref{it:case6_a})(\ref{it:case6_aii}), and therefore $\alpha$ does not have the Borsuk-Ulam property with respect to $\tau_2$ in this case. 
\end{proof}

\section*{Ackowledgements}

This paper was completed during the Postdoctoral Internship of the third author at IME-USP from March 2020 to August 2021. He was supported by Capes/INCTMat project n\textsuperscript{o} 8887.136371/2017-00-465591/2014-0. The first author is partially supported by the Projeto Temático FAPESP, \emph{Topologia Alg\'ebrica, Geom\'etrica e Diferencial}, grant n\textsuperscript{o} 2016/24707-4.

\end{document}